\documentclass[11pt,english]{article}
\usepackage[charter]{mathdesign}
\usepackage[T1]{fontenc}
\usepackage[latin9]{inputenc}
\usepackage[active]{srcltx}
\usepackage{color}
\usepackage{babel}
\usepackage{verbatim}
\usepackage{amsmath}
\usepackage{amsthm}
\usepackage{microtype}
\usepackage[unicode=true,
 bookmarks=true,bookmarksnumbered=false,bookmarksopen=false,
 breaklinks=false,pdfborder={0 0 1},backref=slide,colorlinks=true]
 {hyperref}
\hypersetup{
 citecolor=blue,linktoc=all}

\makeatletter
\theoremstyle{plain}
\newtheorem{thm}{\protect\theoremname}
\theoremstyle{plain}
\newtheorem{lem}[thm]{\protect\lemmaname}
\theoremstyle{definition}
\newtheorem{defn}[thm]{\protect\definitionname}
\theoremstyle{plain}
\newtheorem{cor}[thm]{\protect\corollaryname}

\@ifundefined{date}{}{\date{}}
\usepackage{babel}

\usepackage{color}
\usepackage{babel}

\usepackage{color}

\usepackage{fullpage}
\newcommand{\suppress}[1]{}

\DeclareMathOperator{\re}{Re}

\providecommand{\corollaryname}{Corollary}
\providecommand{\definitionname}{Definition}
\providecommand{\lemmaname}{Lemma}
\providecommand{\theoremname}{Theorem}

\makeatother

\providecommand{\corollaryname}{Corollary}
\providecommand{\definitionname}{Definition}
\providecommand{\lemmaname}{Lemma}
\providecommand{\theoremname}{Theorem}

\begin{document}
\title{Edge Expansion and Spectral Gap of Nonnegative Matrices}
\author{Jenish C. Mehta \\ California Institute of Technology \\ jenishc@gmail.com \and Leonard J. Schulman \\ California Institute of Technology \\ schulman@caltech.edu}
\maketitle
\begin{abstract}
The classic graphical Cheeger inequalities state that if $M$ is an
$n\times n$ \emph{symmetric} doubly stochastic matrix, then 
\[
\frac{1-\lambda_{2}(M)}{2}\leq\phi(M)\leq\sqrt{2\cdot(1-\lambda_{2}(M))}
\]
where $\phi(M)=\min_{S\subseteq[n],|S|\leq n/2}\left(\frac{1}{|S|}\sum_{i\in S,j\not\in S}M_{i,j}\right)$
is the edge expansion of $M$, and $\lambda_{2}(M)$ is the second
largest eigenvalue of $M$. We study the relationship between $\phi(A)$
and the spectral gap $1-\re\lambda_{2}(A)$ for \emph{any} doubly
stochastic matrix $A$ (not necessarily symmetric), where $\lambda_{2}(A)$
is a nontrivial eigenvalue of $A$ with maximum real part. Fiedler
showed that the upper bound on $\phi(A)$ is unaffected, i.e., $\phi(A)\leq\sqrt{2\cdot(1-\re\lambda_{2}(A))}$.
With regards to the lower bound on $\phi(A)$, there are known constructions
with 
\[
\phi(A)\in\Theta\left(\frac{1-\re\lambda_{2}(A)}{\log n}\right),
\]
indicating that at least a mild dependence on $n$ is necessary to
lower bound $\phi(A)$.

In our first result, we provide an \emph{exponentially} better construction
of $n\times n$ doubly stochastic matrices $A_{n}$, for which

\[
\phi(A_{n})\leq\frac{1-\re\lambda_{2}(A_{n})}{\sqrt{n}}.
\]
In fact, \emph{all} nontrivial eigenvalues of our matrices are $0$,
even though the matrices are highly \emph{nonexpanding}. We further
show that this bound is in the correct range (up to the exponent of
$n$), by showing that for any doubly stochastic matrix $A$, 
\[
\phi(A)\geq\frac{1-\re\lambda_{2}(A)}{35\cdot n}.
\]
As a consequence, unlike the symmetric case, there is a (necessary)
loss of a factor of $n^{\alpha}$ for $\frac{1}{2}\leq\alpha\leq1$
in lower bounding $\phi$ by the spectral gap in the nonsymmetric
setting.

Our second result extends these bounds to general matrices $R$ with
nonnegative entries, to obtain a two-sided \emph{gapped} refinement
of the Perron-Frobenius theorem. Recall from the Perron-Frobenius
theorem that for such $R$, there is a nonnegative eigenvalue $r$
such that all eigenvalues of $R$ lie within the closed disk of radius
$r$ about $0$. Further, if $R$ is irreducible, which means $\phi(R)>0$
(for suitably defined $\phi$), then $r$ is positive and all other
eigenvalues lie within the \textit{open} disk, so (with eigenvalues
sorted by real part), $\re\lambda_{2}(R)<r$. An extension of Fiedler's
result provides an upper bound and our result provides the corresponding
lower bound on $\phi(R)$ in terms of $r-\re\lambda_{2}(R)$, obtaining
a two-sided quantitative version of the Perron-Frobenius theorem. 
\end{abstract}
\pagebreak{}

\section{Introduction}

\subsection{Motivation and main result}

We study the relationship between edge expansion and second eigenvalue
of nonnegative matrices. We restrict to doubly stochastic matrices
for exposition in the introduction, since the definitions are simpler
and it captures most of the key ideas. The extension to general nonnegative
matrices is treated in Section \ref{subsec:genl-nonneg}. Let $A$
be an $n\times n$ doubly stochastic matrix, equivalently interpreted
as a bi-regular weighted digraph. The \emph{edge expansion }of $A$,
denoted as $\phi(A)$, is defined as

\[
\phi(A)=\min_{S\subseteq[n],|S|\leq n/2}\frac{\sum\limits _{i\in S,j\not\in S}A_{i,j}}{|S|}.
\]
The fact that $A$ is doubly stochastic implies that $\phi(A)=\phi(A^{T})$.
$\phi$ is a measure of \emph{how much} the graph would have to be
modified for it to lose strong-connectedness; it also lower bounds
\emph{how frequently}, in steady state, the associated Markov chain
switches between the blocks of any bi-partition; thus, it fundamentally
expresses the \emph{extent} to which the graph is connected.

The connection between edge expansion and the second eigenvalue has
been of central importance in the case of \emph{symmetric} doubly
stochastic matrices $M$ (equivalently, reversible Markov chains with
uniform stationary distribution). For such $M$, let the eigenvalues
be $1=\lambda_{1}\geq\lambda_{2}\geq\ldots\geq\lambda_{n}\geq-1$.
The Cheeger inequalities give two-sided bounds between the edge expansion
$\phi(M)$ and the \emph{spectral gap} $1-\lambda_{2}(M)$. 
\begin{thm}
\emph{\label{thm:cheeg}\cite{dodziuk84,AlonM85,alon86}} \emph{(Cheeger's
inequalities for symmetric doubly stochastic matrices)} Let $M$ be
a \emph{symmetric} doubly stochastic matrix, then 
\[
\frac{1-\lambda_{2}(M)}{2}\ \overset{}{\leq}\ \phi(M)\ \overset{}{\leq}\ \sqrt{2\cdot(1-\lambda_{2}(M))}.
\]
\end{thm}

This is closely related to earlier versions for Riemannian manifolds
\cite{cheeger70,Buser82}. Notably, the inequalities in Theorem \ref{thm:cheeg}
do not depend on $n$. Further, they are tight up to constants --
the upper bound on $\phi$ is achieved by the cycle and the lower
bound by the hypercube.

The key question we address in this work is whether or to what extent
the Cheeger inequalities survive for \textit{nonsymmetric} doubly
stochastic matrices (the question was already asked, for e.g., in
\cite{MonT06}). Let $A$ be a doubly stochastic matrix, not necessarily
symmetric. The eigenvalues of $A$ lie in the unit disk around the
origin in the complex plane, with an eigenvalue $1$ called the \emph{trivial}
or \emph{stochastic} eigenvalue corresponding to the eigenvector $\mathbf{1}$
(the all $1$'s vector), and all other eigenvalues considered \emph{nontrivial}.
Let the eigenvalues of $A$ be ordered so that $1=\lambda_{1}\geq\re\lambda_{2}\geq\ldots\geq\re\lambda_{n}\geq-1$.
The spectral gap will be defined as $1-\re\lambda_{2}(A)$. There
are three motivations for this definition: first, the continuous-time
Markov chain based on $A$ is $\exp(t\cdot(A-I))$, and this spectral
gap specifies the largest norm of any of the latter's nontrivial eigenvalues;
second, $\re\lambda_{2}(A)=1$ if and only if $\phi=0$ (i.e., if
the matrix is reducible); finally, this gap lower bounds the distance
(in the complex plane) between the trivial and any nontrivial eigenvalue.

It was noted by Fiedler \cite{fiedler1995estimate} that the upper
bound on $\phi$ in Cheeger's inequality (Theorem \ref{thm:cheeg})
carries over easily to general doubly stochastic matrices, because
for $M=(A+A^{T})/2$, $\phi(A)=\phi(M)$ and $\re\lambda_{2}(A)\leq\lambda_{2}(M)$.
(In fact, Fiedler made a slightly different conclusion with these
observations, but they immediately give the upper bound on $\phi$,
see Appendix \ref{sec:Proof-of-gen-fied} for an extension and proof). 
\begin{lem}
\emph{\label{lem:fied} (Fiedler \cite{fiedler1995estimate})} Let
$A$ be \emph{any} $n\times n$ doubly stochastic matrix, then 
\[
\phi(A)\leq\sqrt{2\cdot(1-\re\lambda_{2}(A))}.
\]
\end{lem}

It remained to investigate the other direction, i.e., the lower bound
on $\phi(A)$ in terms of the spectral gap $1-\re\lambda_{2}(A)$.
Towards this end, we define the function 
\begin{defn}
\label{def:gamma} $\Gamma(n)=\min\left\{ \dfrac{\phi(A)}{1-\re\lambda_{2}(A)}:\ A\ \text{is an \ensuremath{n\times n} doubly stochastic matrix}\right\} .$ 
\end{defn}

For symmetric doubly stochastic matrices this minimum is no less than
$1/2$. However, for doubly stochastic matrices $A$ that are not
necessarily symmetric, there are known (but perhaps not widely known)
examples demonstrating that it is impossible to have a function $\Gamma$
entirely independent of $n$. These examples, discussed in Section
\ref{subsec:Known-Constructions}, show that 
\begin{equation}
\Gamma(n)\leq\dfrac{1}{\log n}.\label{eq:intro-gamm-log-bound}
\end{equation}

One reason that $\phi$ and $1-\re\lambda_{2}$ are important is their
connection to mixing time $\tau$ -- the number of steps after which
a random walk starting at any vertex converges to the uniform distribution
over the vertices. For the case of symmetric doubly stochastic matrices
-- or in general reversible Markov chains -- it is simple to show
that $\tau\in O\left(\frac{\log n}{1-\lambda_{2}}\right)$, and by
Cheeger's inequality (Theorem \ref{thm:cheeg}), it further gives
$\tau\in O\left(\frac{\log n}{\phi^{2}}\right)$ where $n$ is the
number of vertices. For the case of general doubly stochastic matrices
-- or in general not-necessarily-reversible Markov chains -- it
still holds that $\tau\in O\left(\frac{\log n}{\phi^{2}}\right)$
by a result of Mihail (\cite{mihail1989conductance}, and see \cite{fill1991eigenvalue}).
Depending on the situation, either $\re\lambda_{2}$ or $\phi$ may
be easier to estimate. Most often, one is interested in concisely-specified
chains on exponential-size sets, i.e., the number of vertices is $n$
but the complexity parameter is $\log n$. In this case, either $\phi$
or $\lambda_{2}$ (for the reversible case) can be used to estimate
$\tau$. However, if the matrix or chain is given explicitly, then
one reason the Cheeger inequalites are useful is because it is simpler
to estimate $\tau$ using $\lambda_{2}$ which is computable in P
while computing $\phi$ is NP-hard \cite{garey1974some}.

From the point of view of applications to Markov Chain Monte Carlo
(MCMC) algorithms, a $\log n$ loss in the relation between $\phi$
and $1-\re\lambda_{2}$ as implied by \eqref{eq:intro-gamm-log-bound},
is not a large factor. For MCMC algorithms, since ``$n$'' is the
size of the state space being sampled or counted and the underlying
complexity parameter is $\log n$, if it were true that $\Gamma(n)\geq\log^{-c}n$
for some constant $c$, then the loss in mixing time estimates would
be polynomial in the complexity parameter, and thus, the quantity
$1-\re\lambda_{2}$ could still be used to obtain a reasonable estimate
of the mixing time even in the case of nonsymmetric doubly stochastic
matrices.

However, the truth is much different. Our main result is that $\Gamma(n)$
does not scale as $\log^{-c}n$, but is \emph{exponentially} smaller. 
\begin{thm}
\label{thm:main} \emph{(Bounds on $\Gamma$)} 
\[
\dfrac{1}{35\cdot n}\leq\Gamma(n)\leq\dfrac{1}{\sqrt{n}}.
\]
\end{thm}

We give an \emph{explicit} construction of doubly stochastic matrices
$A_{n}$ for the upper bound on $\Gamma$. This construction of highly
nonexpanding doubly stochastic matrices has, in addition, the surprising
property that \emph{every} nontrivial eigenvalue is $0$. Thus, for
non-reversible Markov chains, the connection between $\phi$ and $\re\lambda_{2}$
breaks down substantially, in that the upper bound on mixing time
obtained by lower bounding $\phi$ by the spectral gap $1-\re\lambda_{2}$
can be be exponentially weaker (when the complexity parameter is $\log n$)
than the actual mixing time, whereas for reversible chains there is
at worst a quadratic loss.

This theorem has a very natural extension to general nonnegative matrices,
as we next describe.

\subsection{\label{subsec:genl-nonneg}General nonnegative matrices and a two-sided
quantitative refinement of the Perron-Frobenius theorem}

We extend our results to general nonnegative matrices $R$. By the
Perron-Frobenius theorem (see Theorem \ref{thm:(Perron-Frobenius)}),
since $R$ is nonnegative, it has a nonnegative eigenvalue $r$ (called
the PF eigenvalue) that is also largest in magnitude amongst all eigenvalues,
and the corresponding left and right eigenvectors $u$ and $v$ have
all nonnegative entries. Further, if $R$ is irreducible, i.e., the
underlying weighted digraph on edges with positive weight is strongly
connected (for every $(i,j)$ there is a $k$ such that $R^{k}(i,j)>0$),
then $r$ is a simple eigenvalue, and $u$ and $v$ have all positive
entries. We henceforth assume that nonzero nonnegative $R$ has been
scaled (to $\frac{1}{r}R$) so that $r=1$. Thus we can again write
the eigenvalues of $R$ as $1=\lambda_{1}(R)\geq\re\lambda_{2}(R)\geq\cdots\geq\re\lambda_{n}(R)\geq-1$.
Scaling changes the spectrum but not the edge expansion (which we
define next), so this canonical scaling is necessary before the two
quantities can be compared.

Defining the edge expansion of $R$ is slightly delicate so we explain
the reasoning after Definition \ref{def:gen-edge-expansion}, and
present only the final definition (see Definition \ref{def:gen-edge-expansion})
here. Consider $u$ and $v$ as left and right eigenvectors for eigenvalue
1 of $R$. If \emph{every} such pair $(u,v)$ for $R$ has some $i$
such that $u(i)=0$ or $v(i)=0$, then define $\phi(R)=0$. Otherwise,
let $u$ and $v$ be some positive eigenvectors for eigenvalue 1 of
$R$, normalized so that $\langle u,v\rangle=1$, and define the edge
expansion of $R$ as 
\[
\phi(R)=\min_{S\subseteq[n],\sum_{i\in S}u_{i}\cdot v_{i}\leq\frac{1}{2}}\dfrac{\sum\limits _{i\in S,j\in\overline{S}}R_{i,j}\cdot u_{i}\cdot v_{j}}{\sum\limits _{i\in S}u_{i}\cdot v_{i}}.
\]
Given this definition of $\phi(R)$, we show the following. 
\begin{thm}
\label{thm:gen-phi-lambda2}Let $R$ be a nonnegative matrix with
PF eigenvalue $1$, and let $u$ and $v$ be any corresponding left
and right eigenvectors. Let $\kappa=\min_{i}u_{i}\cdot v_{i}$, and
if $\kappa>0$, let $u$ and $v$ be normalized so that $\langle u,v\rangle=1$.
Then
\[
\frac{1}{30}\cdot\dfrac{1-\re\lambda_{2}(R)}{n+\ln\left(\frac{1}{\kappa}\right)}\ \overset{}{\leq}\ \phi(R)\ \overset{}{\leq}\ \sqrt{2\cdot\left(1-\re\lambda_{2}(R)\right)}.
\]
\end{thm}

The upper bound in Theorem \ref{thm:main} is a straightforward extension
of Fiedler's bound (Lemma \ref{lem:fied}) based on the above mentioned
definition of $\phi$ (Definition \ref{def:gen-edge-expansion}).
Also note that the lower bound in Theorem \ref{thm:main} can be obtained
by setting $\kappa=\frac{1}{n}$ in Theorem \ref{thm:gen-phi-lambda2}.
The upper bound is proven in Appendix \ref{sec:Proof-of-gen-fied}
and the lower bound is shown in Section \ref{sec:lower-bound-on-phi-gen}.

Since Theorem \ref{thm:gen-phi-lambda2} gives a two-sided relation
between the \emph{second} eigenvalue of nonnegative matrices and their
edge expansion, it gives a two-sided quantitative refinement of the
Perron-Frobenius theorem. Although the Perron-Frobenius theorem implies
that the nontrivial eigenvalues of an irreducible nonnegative matrix
$R$ with PF eigenvalue $1$ have real part strictly less than $1$,
it does not give any concrete separation. Further, it also does not
provide a qualitative (or quantitative) implication in the other direction
-- whether a nonnegative matrix $R$ with all nontrivial eigenvalues
having real part strictly less than 1 \emph{implies} that $R$ is
irreducible. Theorem \ref{thm:gen-phi-lambda2} comes to remedy this
by giving a lower and upper bound on the spectral gap in terms of
$\phi$, a quantitative measure of the irreducibility of $R$. We
are not aware of any previous result of this form.

\subsection{Mixing time}

The third quantity we study is the mixing time of general nonnegative
matrices, and we relate it to their singular values, edge expansion,
and spectral gap. This helps us obtain new bounds on mixing time,
and also obtain elementary proofs for known results. These results
are treated in detail in the second part of the paper, in Section
\ref{sec:Mix}.

\subsection{Perspectives}

\subsubsection{\label{subsec:Matrix-perturbations}Matrix perturbations}

Let $A$ be a nonnegative matrix with PF eigenvalue 1 and corresponding
positive left and right eigenvector $w$ with $\langle w,w\rangle=1$,
and let $\kappa=\min_{i}w_{i}^{2}$. Given such $A$, it is certainly
easier to calculate its spectrum than its edge expansion. However,
in other cases, e.g., if the matrix is implicit in a nicely structured
randomized algorithm (as in the canonical paths method \cite{SinclairJ89}),
the edge expansion may actually be easier to bound. From this point
of view, a lower bound on $\phi(A)$ in terms of the spectral gap
is an \textit{eigenvalue perturbation bound.} Specifically, one might
write a nonnegative matrix $A$ with small edge expansion $\phi(A)$
as a perturbation of another nonnegative matrix $A_{0}$, i.e., 
\[
A=A_{0}+\delta\cdot B
\]
where $A_{0}$ has disconnected components $S,S^{c}$ (for $S$ achieving
$\phi(A)$), and $B$ is a matrix such that $\|B\|_{2}\leq1$, $Bw=0$
and $B^{T}w=0$. Due to the conditions on $A$ and $B$, $A_{0}$
has PF eigenvalue 1 with left and right eigenvector $w$, and since
$BD_{w}\mathbf{1}=0$, writing $\mathbf{1}_{\overline{S}}=\mathbf{1}-\mathbf{1}_{S}$,
we have, 
\[
\left|\frac{\langle\mathbf{1}_{S},D_{w}BD_{w}\mathbf{1}_{\overline{S}}\rangle}{\langle\mathbf{1}_{S},D_{w}D_{w}\mathbf{1}_{S}\rangle}\right|=\left|\frac{\langle\mathbf{1}_{S},D_{w}BD_{w}\mathbf{1}_{S}\rangle}{\langle\mathbf{1}_{S},D_{w}D_{w}\mathbf{1}_{S}\rangle}\right|=\left|\frac{\langle D_{w}\mathbf{1}_{S},BD_{w}\mathbf{1}_{S}\rangle}{\langle D_{w}\mathbf{1}_{S},D_{w}\mathbf{1}_{S}\rangle}\right|\leq\|B\|_{2}\leq1,
\]
and it follows then that $\phi(R_{0})-\delta\leq\phi(R)\leq\phi(R_{0})+\delta$,
and since in this case $\phi(R_{0})=0$, so $\phi(A)\leq\delta.$
Edge expansion is therefore stable with respect to perturbation by
$B$. What about the spectral gap?

$A_{0}$ has (at least) a double eigenvalue at 1, and $A$ retains
a simple eigenvalue at $1$, so it is natural to try to apply eigenvalue
stability results, specifically Bauer-Fike (\cite{bauer1960norms},
\cite{Bhatia97} \S VIII), to obtain an upper bound on $|1-\lambda_{2}(A)|$
(and therefore also on $1-\re\lambda_{2}(A)$). However, Bauer-Fike
requires $A_{0}$ to be diagonalizable, and the quality of the bound
depends upon the condition number of the diagonalizing change of basis\footnote{The condition number is $\inf(\|B\|\cdot\|B^{-1}\|)$ over $B$ such
that $BAB^{-1}$ is diagonal, with $\|\cdot\|$ being operator norm.}. There are extensions of Bauer-Fike which do not require diagonalizability,
but deteriorate exponentially in the size of the Jordan block, and
the bound still depends on the condition number of the (now Jordan-izing)
change of basis (\cite{Saad11} Corollary 3.2). Since there is no
a priori (i.e., function of $n$) bound on these condition numbers,
these tools unfortunately do not imply any, even weak, result analogous
to Theorem \ref{thm:main}.

In summary, the lower bound in Theorem \ref{thm:main} should be viewed
as a new eigenvalue perturbation bound: 
\[
1-\re\lambda_{2}(A)\leq30\cdot\delta\cdot\left(n+\ln\left(\frac{1}{\kappa}\right)\right).
\]
A novel aspect of this bound, in comparison with the prior literature
on eigenvalue perturbations, is that it does not depend on the condition
number (or even the diagonalizability) of $A$ or of $A_{0}$.

\subsubsection{Relating eigenvalues of a nonnegative matrix and its additive symmetrization}

Let $A$ be any nonnegative matrix with PF eigenvalue 1, and corresponding
positive left and right eigenvector $w$. Our result helps to give
the following bounds between the second eigenvalue of a nonnegative
matrix $A$ and that of its \textit{additive symmetrization} $\frac{1}{2}(A+A^{T})$. 
\begin{lem}
\label{lem:bound-eig-A-and_AplusAT}Let $A$ be any nonnegative matrix
with PF eigenvalue 1 and corresponding positive left and right eigenvector
$w$, let $\kappa=\min_{i}w_{i}^{2}$ and $M=\frac{1}{2}(A+A^{T})$.Then
\[
\frac{1}{1800}\left(\frac{1-\re\lambda_{2}(A)}{n+\ln\left(\frac{1}{\kappa}\right)}\right)^{2}\overset{}{\leq}1-\lambda_{2}\left(M\right)\overset{\text{\text{}}}{\leq}1-\re\lambda_{2}(A).
\]
\end{lem}

The bounds immediately follows from Theorem \ref{thm:gen-phi-lambda2}
and the fact that $\phi$ is unchanged by additive symmetrization.
We remark that any improved lower bound on $1-\lambda_{2}(M)$ in
terms of $1-\re\lambda_{2}(A)$ will help to improve the lower bound
in Theorem \ref{thm:gen-phi-lambda2}. As a consequence of Lemma \ref{lem:bound-eig-A-and_AplusAT},
any bound based on the second eigenvalue of symmetric nonnegative
matrices can be applied, with dimension-dependent loss, to nonnegative
matrices that have identical left and right eigenvector for the PF
eigenvalue.

An example application of Lemma \ref{lem:bound-eig-A-and_AplusAT}
is the following. For some doubly stochastic $A$ (not necessarily
symmetric), consider the continuous time Markov Chain associated with
it, $\exp(t\cdot(A-I))$. It is well-known (for instance, \cite{diaconis1996logarithmic})
that for any standard basis vector $x_{i}$,

\[
\left\Vert \exp(t\cdot(A-I))x_{i}-\frac{1}{n}\mathbf{1}\right\Vert _{1}^{2}\leq n\cdot\exp\left(-2\cdot\left(1-\lambda_{2}\left(M\right)\right)\cdot t\right).
\]
Thus, using Lemma \ref{lem:bound-eig-A-and_AplusAT}, we immediately
get the bound in terms of the second eigenvalue of $A$ itself (instead
of its additive symmetrization), 
\[
\left\Vert \exp(t\cdot(A-I))x_{i}-\frac{1}{n}\mathbf{1}\right\Vert _{1}^{2}\leq n\cdot\exp\left(-\frac{1}{1800}\left(\frac{1-\re\lambda_{2}(A)}{n+\ln\left(\frac{1}{\kappa}\right)}\right)^{2}\cdot t\right).
\]

\subsubsection{\label{subsec:The-role-of-singular}The role of singular values}

Although in the symmetric case singular values are simply the absolute
values of the eigenvalues, the two sets can be much less related in
the nonsymmetric case. It is not difficult to show the following (see
Appendix \ref{sec:Proof-of-sigmabound}). 
\begin{lem}
\label{lem:sigma-bound-phi} Let $A$ be a nonnegative matrix with
PF eigenvalue 1, and let $w$ be the positive vector such that $Aw=w$
and $A^{T}w=w$. Then 
\[
\frac{1-\sigma_{2}(A)}{2}\leq\phi(A),
\]
where $\sigma_{2}(A)$ is the second largest singular value of $A$. 
\end{lem}

\begin{proof}
The proof is given in Appendix~\ref{sec:Proof-of-sigmabound}. 
\end{proof}
Despite appearances, this tells us little about edge expansion\suppress{\footnote{Lazy-fying does not help since it can change singular values of nonsymmetric
matrices sharply (unlike eigenvalues), so it would not be possible
to obtain an estimate on $\phi$ in terms of the singular values of
the original matrix $A$.}}. To see this, consider the directed cycle on $n$ vertices, for
which every singular value (and in particular $\sigma_{2}$) is 1,
so Lemma \ref{lem:sigma-bound-phi} gives a lower bound of 0 for $\phi$,
although $\phi=2/n$. A meaningful lower bound should be 0 if and
only if the graph is disconnected, i.e. it should be continuous in
$\phi$. An even more striking example is that of de Bruijn graphs
(described in Section \ref{subsec:Known-Constructions}), for which
half the singular values are 1, although $\phi=\Theta(1/\log n)$.
Eigenvalues, on the other hand, are more informative, since a nonnegative
matrix with PF eigenvalue 1 has a multiple \emph{eigenvalue} at $1$
if and only if, as a weighted graph, it is disconnected.

Despite these observations, singular values can be useful to infer
information about mixing time, and can be used to recover all known
upper bounds on mixing time using $\phi$, as discussed in Section
\ref{sec:Mix} and Lemma \ref{lem:mixt-singval}. \\

Outline: We state the preliminary definitions, theorems, and notations
in Section \ref{sec:Preliminaries}. We give the construction for
the upper bound on $\Gamma$ in Theorem \ref{thm:main} in Section
\ref{sec:Construction-of-non-expanding}, and the lower bound on $\Gamma$
will follow from the general lower bound on $\phi$ in Theorem \ref{thm:gen-phi-lambda2}.
We show the upper bound on $\phi$ in Theorem \ref{thm:gen-phi-lambda2}
in Appendix \ref{sec:Proof-of-gen-fied}, and show the lower bound
on $\phi$ in Section \ref{sec:lower-bound-on-phi-gen}. We relate
mixing time to singular values, edge expansion, and the spectral gap
in Section \ref{sec:Mix} respectively. We defer all proofs to the
Appendix.

\section{\label{sec:Preliminaries}Preliminaries}

We consider \emph{doubly stochastic }matrices $A\in\mathbb{R}^{n\times n}$
which are nonnegative matrices with entries in every row and column
summing to $1$. We also consider general nonnegative matrices $R$,
and say that $R$ is \emph{strongly connected} or \emph{irreducible},
if there is a path from $s$ to $t$ for every pair $(s,t)$ of vertices
in the underlying digraph on edges with positive weight, i.e. for
every $(s,t)$ there exists $k>0$ such that $R^{k}(s,t)>0$. We say
$R$ is \emph{weakly connected}, if there is a pair of vertices $(s,t)$
such that there is a path from $s$ to $t$ but no path from $t$
to $s$ in the underlying digraph (on edges with positive weight).
We restate the Perron-Frobenius theorem for convenience. 
\begin{thm}
\emph{\label{thm:(Perron-Frobenius)}(Perron-Frobenius theorem \cite{perron1907theorie,frobenius1912matrizen})}
Let $R\in\mathbb{R}^{n\times n}$ be a nonnegative matrix. Then the
following hold for $R$. 
\begin{enumerate}
\item $R$ has some nonnegative eigenvalue $r$, such that all other eigenvalues
have magnitude at most $r$, and $R$ has nonnegative left and right
eigenvectors $u$ and $v$ for $r$. 
\item If $R$ has some \emph{positive} left and right eigenvectors $u$
and $v$ for some eigenvalue $\lambda$, then $\lambda=r$. 
\item If $R$ is irreducible, then $r$ is positive and simple (unique),
$u$ and $v$ are positive and unique, and all other eigenvalues have
real part strictly less than $r$. 
\end{enumerate}
\end{thm}

We denote the all $1$'s vector by $\boldsymbol{1}$, and note that
for a doubly stochastic matrix $A$, $\boldsymbol{1}$ is both the
left and right eigenvector with eigenvalue 1. We say that $1$ is
the \emph{trivial (or stochastic) eigenvalue}, and $\boldsymbol{1}$
is the \emph{trivial (or stochastic) eigenvector, }of $A$. All other
eigenvalues of $A$ (corresponding to eigenvectors that are not trivial)
will be called \emph{nontrivial} \emph{(or nonstochastic) eigenvalues
}of $A$. Similarly, by Perron-Frobenius (Theorem \ref{thm:(Perron-Frobenius)},
part 1), a nonnegative matrix $R$ will have a simple nonnegative
eigenvalue $r$ such that all eigenvalues have magnitude at most $r$,
and it will be called the \emph{trivial} or PF eigenvalue of $R$,
and all other eigenvalues of $R$ will be called \emph{nontrivial.
} The left and right eigenvectors corresponding to $r$ will be called
the \emph{trivial or PF left eigenvector} and \emph{trivial or PF
right eigenvector}. This leads us to the following definition. 
\begin{defn}
\emph{(Second eigenvalue)} If $A$ is a doubly stochastic matrix,
then $\lambda_{2}(A)$ is the nontrivial eigenvalue of $A$ with the
maximum real part, and $\lambda_{m}(A)$ is the nontrivial eigenvalue
that is largest in magnitude. Similarly, if $R$ is any general nonnegative
matrix with PF eigenvalue 1, then $\lambda_{2}(R)$ is the nontrivial
eigenvalue with the maximum real part, i.e., $1=\lambda_{1}(R)\geq\re\lambda_{2}(R)\geq\cdots\geq\re\lambda_{n}(R)\geq-1$,
and $\lambda_{m}(R)$ is the nontrivial eigenvalue that is largest
in magnitude. 
\end{defn}

We will also consider singular values of nonnegative matrices $A$
with identical positive left and right eigenvector $w$ for PF eigenvalue
1, and denote them as $1=\sigma_{1}(A)\geq\sigma_{2}(A)\geq\cdots\geq\sigma_{n}(A)\geq0$
(see Lemma \ref{lem:transRtoA} for proof of $\sigma_{1}(A)=1$).
We denote $(i,j)$'th entry of $M\in\mathbb{C}^{n\times n}$ by $M(i,j)$
or $M_{i,j}$ depending on the importance of indices in context, denote
the conjugate-transpose of $M$ as $M^{*}$ and the transpose of $M$
as $M^{T}$. Any $M\in\mathbb{C}^{n\times n}$ has a \emph{Schur decomposition}
(see, e.g., \cite{lax07}) $M=UTU^{*}$ where $T$ is an upper triangular
matrix whose diagonal entries are the eigenvalues of $M$, and $U$
is a unitary matrix. When we write ``vector'' we mean by default
a column vector. For a vector $v$, we write $v(i)$ or $v_{i}$ to
denote its $i$'th entry. For any two vectors $x,y\in\mathbb{C}^{n}$,
we use the standard \emph{inner product }$\langle x,y\rangle=\sum_{i=1}^{n}x_{i}^{*}\cdot y_{i}$
defining the norm $\|x\|_{2}=\sqrt{\langle x,x\rangle}$. We write
$u\perp v$ to indicate that $\langle u,v\rangle=0$. Note that $\langle x,My\rangle=\langle M^{*}x,y\rangle$.
We denote the operator norm of $M$ by $\|M\|_{2}=\max_{u:\|u\|_{2}=1}\|Mu\|_{2}$,
and recall that the operator norm is at most the Frobenius norm, i.e.,
$\|M\|_{2}\leq\sqrt{\sum_{i,j}|M_{i,j}|^{2}}.$ We write $D_{u}$
for the diagonal matrix whose diagonal contains the vector $u$. Recall
the Courant-Fischer variational characterization of eigenvalues for
symmetric real matrices, applied to the second eigenvalue: 
\[
\max_{u\perp v_{1}}\frac{\langle u,Mu\rangle}{\langle u,u\rangle}=\lambda_{2}(M),
\]
where $v_{1}$ is the eigenvector for the largest eigenvalue of $M$.
We will use the symbol $J$ for the all $1$'s matrix divided by $n$,
i.e., $J=\frac{1}{n}\mathbf{1}\cdot\mathbf{1}^{T}$.

We say that any subset $S\subseteq[n]$ is a \emph{cut}, denote its
complement by $\overline{S}$, and denote the \emph{characteristic
vector of a cut} as $\boldsymbol{1}_{S}$, where $\boldsymbol{1}_{S}(i)=1$
if $i\in S$ and 0 otherwise. 
\begin{defn}
\emph{\label{def:(Edge-expansion-of}(Edge expansion of doubly stochastic
matrices)} For a doubly stochastic matrix $A$, the \emph{edge expansion
of the cut $S$} is defined as 
\[
\phi_{S}(A):=\dfrac{\langle\mathbf{1}_{S},A\mathbf{1}_{\overline{S}}\rangle}{\min\{\langle\mathbf{1}_{S},A\mathbf{1}\rangle,\langle\mathbf{1}_{\overline{S}},A\mathbf{1}\rangle\}}
\]
and the \emph{edge expansion of $A$} is defined as 
\[
\phi(A)=\min_{S\subseteq[n]}\phi_{S}(A)=\min_{S\subseteq[n]}\dfrac{\langle\mathbf{1}_{S},A\mathbf{1}_{\overline{S}}\rangle}{\min\{\langle\mathbf{1}_{S},A\mathbf{1}\rangle,\langle\mathbf{1}_{\overline{S}},A\mathbf{1}\rangle\}}=\min_{S,|S|\leq n/2}\frac{\sum\limits _{i\in S,j\in\overline{S}}A_{i,j}}{|S|}.
\]
\end{defn}

We wish to extend these notions to general nonnegative matrices $R$.
Since eigenvalues and singular values of real matrices remain unchanged
whether we consider $R$ or $R^{T}$, the same should hold of a meaningful
definition of edge expansion. However, note that Definition \ref{def:(Edge-expansion-of}
has this independence only if the matrix is Eulerian, i.e., $R\mathbf{1}=R^{T}\mathbf{1}$.
Thus, to define edge expansion for general matrices, we transform
$R$ using its left and right eigenvectors $u$ and $v$ for eigenvalue
1 to obtain $D_{u}RD_{v}$, which is indeed Eulerian, since 
\[
D_{u}RD_{v}\mathbf{1}=D_{u}Rv=D_{u}v=D_{u}D_{v}\mathbf{1}=D_{v}D_{u}\mathbf{1}=D_{v}u=D_{v}R^{T}u=D_{v}R^{T}D_{u}\mathbf{1}.
\]
Since $D_{u}RD_{v}$ is Eulerian, we can define the edge expansion
of $R$ similar to that for doubly stochastic matrices:
\begin{defn}
\label{def:gen-edge-expansion} \emph{(Edge expansion of nonnegative
matrices)} Let $R\in\mathbb{R}^{n\times n}$ be a nonnegative matrix
with PF eigenvalue $1$. If there are no positive (i.e., \textit{everywhere}
positive) left and right eigenvectors $u$ and $v$ for eigenvalue
1, then define the edge expansion $\phi(R)=0$. Else, let $u$ and
$v$ be any (see Lemma \ref{lem:phiR-ind-uv} for justification) positive
left and right eigenvectors for eigenvalue 1, normalized so that $\langle u,v\rangle=1$.
The \emph{edge expansion of the cut $S$} is defined as 
\begin{equation}
\phi_{S}(R):=\dfrac{\langle\mathbf{1}_{S},D_{u}RD_{v}\mathbf{1}_{\overline{S}}\rangle}{\min\{\langle\mathbf{1}_{S},D_{u}RD_{v}\mathbf{1}\rangle,\langle\mathbf{1}_{\overline{S}},D_{u}RD_{v}\mathbf{1}\rangle\}}\label{phiSR}
\end{equation}
and the \emph{edge expansion of $R$} is defined as 
\[
\phi(R)=\min_{S\subseteq[n]}\phi_{S}(R)=\min_{S\subseteq[n],\sum_{i\in S}u_{i}\cdot v_{i}\leq\frac{1}{2}}\dfrac{\langle1_{S},D_{u}RD_{v}1_{\overline{S}}\rangle}{\langle1_{S},D_{u}D_{v}\mathbf{1}\rangle}=\min_{S\subseteq[n],\sum_{i\in S}u_{i}\cdot v_{i}\leq\frac{1}{2}}\dfrac{\sum\limits _{i\in S,j\in\overline{S}}R_{i,j}\cdot u_{i}\cdot v_{j}}{\sum\limits _{i\in S}u_{i}\cdot v_{i}}.
\]
\end{defn}

\begin{lem}
\label{lem:phiR-ind-uv}Let $R\in\mathbb{R}^{n\times n}$ be a nonnegative
matrix with PF eigenvalue $1$. Then the value of $\phi(R)$ according
to Definition \ref{def:gen-edge-expansion} is independent of the
choice of the specific (left and right) eigenvectors $u$ and $v$
for eigenvalue 1 of $R$. 
\end{lem}

\begin{proof}
Let $G$ be the underlying unweighted directed graph for $R$, where
there is an edge $(u,v)$ in $G$ if and only if $R_{u,v}>0$. We
prove the lemma based on the structure of $G$. Let $G$ be maximally
partitioned into $k$ weakly connected components. 
\begin{enumerate}
\item If $G$ has some weakly connected component which does \emph{not}
have a 1 eigenvalue, then for any pair of left and right eigenvectors
$u$ and $v$ for eigenvalue 1, there will be at least one entry $i$
such that $u_{i}=0$ or $v_{i}=0$. For such matrices $R$, from Definition
\ref{def:gen-edge-expansion}, $\phi(R)=0$. 
\item If all weakly connected components of $G$ have eigenvalue 1, but
there is some weakly connected component $S$ that is not strongly
connected, then there is no positive pair of eigenvectors $u$ and
$v$ for $R$, or even for $S$. Observe that $S=\left[\begin{array}{cc}
A & B\\
0 & C
\end{array}\right]$ with $B\not=0$, else $G$ is not maximally partitioned. For the
sake of contradiction, let $x$ and $y$ be positive left and right
eigenvectors for eigenvalue 1 of $S$. From $Sy=y$ and $S^{T}x=x$,
we get that $Ay_{1}+By_{2}=y_{1}$ and $A^{T}x_{1}=x_{1}$ with $x$
and $y$ partitioned into $x_{1},x_{2}$ and $y_{1},y_{2}$ based
on the sizes of $A,B,C$. Thus, 
\[
\langle x_{1},y_{1}\rangle=\langle x_{1},Ay_{1}+By_{2}\rangle=\langle A^{T}x_{1},y_{1}\rangle+\langle x_{1},By_{2}\rangle=\langle x_{1},y_{1}\rangle+\langle x_{1},By_{2}\rangle
\]
implying $\langle x_{1},By_{2}\rangle=0$ which is a contradiction
since $x_{1}$ and $y_{1}$ are positive and there is some entry in
$B$ which is not 0. Thus, since every pair of eigenvectors $x$ and
$y$ for eigenvalue 1 of $S$ has some entry $i$ with $x_{i}=0$
or $y_{i}=0$, every pair of (left and right) eigenvectors $u$ and
$v$ for $R$ (for eigenvalue 1) has some entry $i$ which is 0, and
so by Definition \ref{def:gen-edge-expansion}, $\phi(R)=0$. 
\item If $G$ consists of \emph{one} strongly connected component (i.e.,
$k=1$), then by Perron-Frobenius (Theorem \ref{thm:(Perron-Frobenius)},
part 3), there is a \emph{unique} (up to scaling) pair of positive
eigenvectors $u$ and and $v$ for eigenvalue 1.
\item The remaining case is that $G$ has $k\geq2$ strongly connected components
each with eigenvalue 1. By Perron-Frobenius (Theorem \ref{thm:(Perron-Frobenius)}),
there is some pair of positive left and right eigenvectors $u$ and
$v$ for eigenvalue 1 (obtained by concatenating the positive left
and right eigenvectors for each component individually). Choose the
set $S$ to be one of the components (the one for which the denominator
of equation \eqref{phiSR} is at most $\frac{1}{2}$), then the numerator
of equation \eqref{phiSR} will be $0$, and thus $\phi(R)=0$ even
in this case, corresponding to the existence of a strict subset of
vertices with no expansion. 
\end{enumerate}
Thus, Definition \ref{def:gen-edge-expansion} for $\phi(R)$ does
not depend on the specific choice of $u$ and $v$. 
\end{proof}
The Perron-Frobenius theorem (Theorem \ref{thm:(Perron-Frobenius)},
part 3) can now be restated in terms of $\phi(R)$ and $\re\lambda_{2}(R)$
as follows. 
\begin{lem}
\label{lem:phi-Perron-Frobenius}(Perron-Frobenius, part 3 of Theorem
\ref{thm:(Perron-Frobenius)}, restated) Let $R$ be a nonnegative
matrix with PF eigenvalue 1. If $\phi(R)>0$, then $\re\lambda_{2}(R)<1$. 
\end{lem}

Further, we obtain the following converse of Lemma \ref{lem:phi-Perron-Frobenius}. 
\begin{lem}
\label{lem:Converse-phi-per-frob}(Converse of Lemma \ref{lem:phi-Perron-Frobenius})
Let $R$ be a nonnegative matrix with PF eigenvalue 1. If $\re\lambda_{2}(R)<1$,
\emph{and} there exists a pair of positive left and right eigenvectors
$u$ and $v$ for eigenvalue 1 of $R$, then $\phi(R)>0$. 
\end{lem}

\begin{proof}
We show the contrapositive. Let $R$ be as stated and let $\phi(R)=0$.
From Definition \ref{def:gen-edge-expansion}, if there are no positive
eigenvectors $u$ and $v$ for eigenvalue 1, the lemma holds. So assume
there are some positive $u$ and $v$ for eigenvalue 1 of $R$. Since
$\phi(R)=0$, there is some set $S$ for which $\phi_{S}(R)=0$, or
$\sum_{i\in S,j\in\overline{S}}R_{i,j}\cdot u_{i}\cdot v_{j}=0$.
But since $u_{i}>0$ and $v_{i}>0$ for each $i$, and since $R$
is nonnegative, it implies that for each $i\in S,j\in\overline{S}$,
$R_{i,j}=0$. Further, since $D_{u}RD_{v}$ is Eulerian, i.e. $D_{u}RD_{v}\mathbf{1}=D_{v}R^{T}D_{u}\mathbf{1}$,
it implies that $\sum_{i\in\overline{S},j\in S}R_{i,j}\cdot u_{i}\cdot v_{j}=0$,
further implying that $R_{i,j}=0$ for each $i\in\overline{S},j\in S$.
As a consequence, $v$ can be rewritten as two vectors $v_{S}$ and
$v_{\overline{S}}$, where $v_{S}(i)=v(i)$ if $i\in S$ and $v_{S}(i)=0$
otherwise, and similarly $v_{\overline{S}}$. Similarly, split $u$
into $u_{S}$ and $u_{\overline{S}}$. Note that $v_{S}$ and $v_{\overline{S}}$
are linearly independent (in fact, orthogonal), and both are right
eigenvectors for eigenvalue 1 (similarly $u_{S}$ and $u_{\overline{S}}$
as left eigenvectors). Thus, this implies that eigenvalue 1 for $R$
has multiplicity at least 2, and thus $\lambda_{2}(R)=1$, as required. 
\end{proof}
The upper and lower bounds on $\phi(R)$ in Theorem \ref{thm:gen-phi-lambda2}
are \emph{quantitative} versions of Lemmas \ref{lem:phi-Perron-Frobenius}
and \ref{lem:Converse-phi-per-frob} respectively.

We note that Cheeger's inequalities hold not only for any symmetric
doubly stochastic matrix, but also for any nonnegative matrix $R$
which satisfies detailed balance. We say that a nonnegative matrix
$R$ with positive left and right eigenvectors $u$ and $v$ for PF
eigenvalue 1 satisfies \emph{detailed balance} if $D_{u}RD_{v}$ is
symmetric, which generalizes the usual definition of detailed balance
(or reversibility) for stochastic matrices. We first note that if
$R$ satisfies the condition of detailed balance, then $R$ has all
real eigenvalues. To see this, let $W=D_{u}^{\frac{1}{2}}D_{v}^{-\frac{1}{2}}$
where the inverses are well-defined since we assume $u$ and $v$
are positive (else $\phi=0$ by definition), and $A=WRW^{-1}$ where
$A$ has same eigenvalues as $R$. For $w=D_{u}^{\frac{1}{2}}D_{v}^{\frac{1}{2}}\mathbf{1}$
which is both the left and right eigenvector of $A$ for eigenvalue
1, the detailed balance condition translates to $D_{w}AD_{w}=D_{w}A^{T}D_{w}$,
which implies $A=A^{T}$, which further implies that all eigenvalues
of $A$ (and $R$) are real. We can thus state Cheeger inequalities
for nonnegative matrices satisfying detailed balance. 
\begin{thm}
\emph{\label{thm:gen-Cheeger's-inequalities}(Cheeger's inequalities
for nonnegative matrices satisfying detailed balance)} Let $R$ be
a nonnegative matrix with PF eigenvalue 1 and positive left and right
eigenvectors $u$ and $v$ (else $\phi(R)=0$ by definition), and
let $R$ satisfy the condition of detailed balance, i.e. $D_{u}RD_{v}=D_{v}R^{T}D_{u}$.
Then 
\[
\frac{1-\lambda_{2}(R)}{2}\leq\phi(R)\leq\sqrt{2\cdot(1-\lambda_{2}(R))}.
\]
\end{thm}

\section{\label{sec:Construction-of-non-expanding}Construction of doubly
stochastic matrices with small edge expansion and large spectral gap}

As discussed, it might have been tempting to think that $\Gamma(n)$
(see Definition \ref{def:gamma}) should be independent of the matrix
size $n$, since this holds for the symmetric case, and also for the
upper bound on $\phi$ for the general nonsymmetric doubly stochastic
case (Lemma \ref{lem:fied}). However, two known examples showed that
a mild dependence on $n$ cannot be avoided.

\subsection{\label{subsec:Known-Constructions}Known Constructions}

\noindent \textbf{Klawe-Vazirani construction:} The first is a construction
of Klawe~\cite{klawe1984limitations} -- these families of $d$-regular
undirected graphs have edge expansion $(\log n)^{-\gamma}$ for various
$0<\gamma<1$. However, there is a natural way in which to direct
the edges of these graphs and obtain a $(d/2)$-in, $(d/2)$-out -regular
graph, and it was noted by U.\ Vazirani \cite{UVaz-pers} in the
1980s that for this digraph $A$, which shares the same edge expansion
as Klawe's, all eigenvalues (except the stochastic eigenvalue) have
norm $\leq1/2$. Specifically, one construction is as follows: let
$n$ be an odd prime, and create the graph on $n$ vertices with in-degree
and out-degree 2 by connecting every vertex $v\in\mathbb{Z}/n$ to
two vertices, $1+v$ and $2v$. Dividing the adjacency matrix of the
graph by 2 gives a doubly stochastic matrix $A_{KV}$. It is simple
to see (by transforming to the Fourier basis over $\mathbb{Z}/n$)
that the characteristic polynomial of $A_{KV}$ is $x(x-1)((2x)^{n-1}-1)/(2x-1)$,
so apart from the trivial eigenvalue 1, $A_{KV}$ has $n-2$ nontrivial
eigenvalues $\lambda$ such that $|\lambda|=\frac{1}{2}$ and one
eigenvalue 0, and thus, $\re\lambda_{2}(A_{KV})\leq\frac{1}{2}$.
Further, upper bounding the edge expansion $\phi$ by the vertex expansion
bound (Theorem 2.1 in \cite{klawe1984limitations}), it follows that
for some constant $c$, 
\[
\Gamma(n)\leq\frac{\phi(A_{KV})}{1-\re\lambda_{2}(A_{KV})}\leq c\cdot\left(\frac{\log\log n}{\log n}\right)^{1/5}.
\]

\noindent \textbf{de Bruijn construction:} A second example is the
de Bruijn digraph \cite{de1946combinatorial}. This is if anything
even more striking: the doubly stochastic matrix (again representing
random walk along an Eulerian digraph) has edge expansion $\Theta(1/\log n)$
\cite{delorme1998spectrum}, yet \emph{all} the nontrivial eigenvalues
are 0. More specifically, define a special case of de Bruijn \cite{de1946combinatorial}
graphs as follows: Let $n=2^{k}$ for some integer $k$, and create
the graph of degree 2 on $n$ vertices by directing edges from each
vertex $v=(v_{1},v_{2},\ldots,v_{k})\in\{0,1\}^{k}$ to two vertices,
$(v_{2},v_{3},\ldots,v_{k},0)$ and $(v_{2},v_{3},\ldots,v_{k},1)$.
Dividing the adjacency matrix of the graph by 2 gives a doubly stochastic
matrix $A_{dB}$. Since this random walk completely forgets its starting
point after $k$ steps, every nontrivial eigenvalue of $A_{dB}$ is
0 (and each of its Jordan blocks is of dimension at most $k$). Further,
it was shown in \cite{delorme1998spectrum} that $\phi(A_{dB})\in\Theta(1/k)$,
and thus, 
\[
\Gamma(n)\leq\frac{\phi(A_{dB})}{1-\re\lambda_{2}(A_{dB})}\leq\frac{1}{k}=\frac{1}{\log n}.
\]

\noindent \textbf{Other literature -- Feng-Li construction:} We round
out this discussion by recalling that Alon and Boppana \cite{alon86,Nilli91}
showed that for any infinite family of $d$-regular undirected graphs,
the adjacency matrices, normalized to be doubly stochastic and with
eigenvalues $1=\lambda_{1}\geq\lambda_{2}\geq\ldots\geq-1$, have
$\lambda_{2}\geq\frac{2\sqrt{d-1}}{d}-o(1)$. Feng and Li \cite{Li92}
showed that undirectedness is essential to this bound: they provide
a construction of cyclically-directed $r$-partite ($r\geq2$) $d$-regular
digraphs (with $n=kr$ vertices for $k>d$, $\gcd(k,d)=1$), whose
normalized adjacency matrices have (apart from $r$ ``trivial''
eigenvalues), only eigenvalues of norm $\leq1/d$. The construction
is of an affine-linear nature quite similar to the preceding two,
and to our knowledge does not give an upper bound on $\Gamma$ any
stronger than those.

\subsection{A new construction}

To achieve the upper bound in Theorem \ref{thm:main}, we need a construction
that is \emph{exponentially} better than known examples. We give an
explicit construction of $n\times n$ doubly stochastic matrices $A_{n}$
(for every $n$) that are highly \emph{nonexpanding}, since they contain
sets with edge expansion less than $1/\sqrt{n}$, even though \emph{every
}nontrivial eigenvalue is 0. The construction might seem nonintuitive,
but in Appendix \ref{sec:Intuition-behind-the-construction} we give
some explanation of how to arrive at it. 
\begin{thm}
\label{thm:constr_main}Let $m=\sqrt{n}$,

\begin{align*}
a_{n} & =\dfrac{m^{2}+m-1}{m\cdot(m+2)}, & b_{n} & =\dfrac{m+1}{m\cdot(m+2)}, & c_{n} & =\dfrac{1}{m\cdot(m+1)},\\
\\
d_{n} & =\dfrac{m^{3}+2m^{2}+m+1}{m\cdot(m+1)\cdot(m+2)}, & e_{n} & =\dfrac{1}{m\cdot(m+1)\cdot(m+2)}, & f_{n} & =\dfrac{2m+3}{m\cdot(m+1)\cdot(m+2)},
\end{align*}

\noindent and define the $n\times n$ matrix

\[
A_{n}=\left[\begin{array}{ccccccccc}
a_{n} & b_{n} & 0 & 0 & 0 & 0 & 0 & \cdots & 0\\
0 & c_{n} & d_{n} & e_{n} & e_{n} & e_{n} & e_{n} & \cdots & e_{n}\\
0 & c_{n} & e_{n} & d_{n} & e_{n} & e_{n} & e_{n} & \cdots & e_{n}\\
0 & c_{n} & e_{n} & e_{n} & d_{n} & e_{n} & e_{n} & \cdots & e_{n}\\
0 & c_{n} & e_{n} & e_{n} & e_{n} & d_{n} & e_{n} & \cdots & e_{n}\\
0 & c_{n} & e_{n} & e_{n} & e_{n} & e_{n} & d_{n} & \cdots & e_{n}\\
\vdots & \vdots & \vdots & \vdots & \vdots & \vdots & \vdots & \ddots & \vdots\\
0 & c_{n} & e_{n} & e_{n} & e_{n} & e_{n} & e_{n} & \ldots & d_{n}\\
b_{n} & f_{n} & c_{n} & c_{n} & c_{n} & c_{n} & c_{n} & \cdots & c_{n}
\end{array}\right].
\]
Then the following hold for $A_{n}$: 
\begin{enumerate}
\item $A_{n}$ is doubly stochastic. 
\item Every nontrivial eigenvalue of $A_{n}$ is 0. 
\item The edge expansion is bounded as 
\[
\dfrac{1}{6\sqrt{n}}\leq\phi(A_{n})\leq\dfrac{1}{\sqrt{n}}.
\]
\item As a consequence of 1,2,3, 
\[
\phi(A_{n})\leq\dfrac{1-\re\lambda_{2}(A_{n})}{\sqrt{n}}
\]
and thus 
\[
\Gamma(n)\leq\dfrac{1}{\sqrt{n}}.
\]
\end{enumerate}
\end{thm}

\begin{proof}
The proof is given in Appendix \ref{sec:Proof-of-construction}. 
\end{proof}
This completes the proof of the upper bound in Theorem \ref{thm:main}.
We remark that for the matrices $A_{n}$ constructed in Theorem \ref{thm:constr_main},
it holds that 
\[
\phi(A_{n})\leq\frac{1-|\lambda_{i}(A)|}{\sqrt{n}}
\]
for any $i\not=1$, giving a stronger guarantee than that required
for Theorem \ref{thm:main}.

We also remark that it would be unlikely to arrive at such a construction
by algorithmic simulation, since the eigenvalues of the matrices $A_{n}$
are extremely sensitive. Although $\lambda_{2}(A_{n})=0$, if we shift
only $O(1/\sqrt{n})$ of the mass in the matrix $A_{n}$ to create
a matrix $A'_{n}$, by replacing $a_{n}$ with $a'_{n}=a_{n}+b_{n}$,
$b_{n}$ with $b'_{n}=0$, $f_{n}$ with $f'_{n}=f_{n}+b_{n}$ and
keeping $c_{n},d_{n},e_{n}$ the same, then $\lambda_{2}(A'_{n})=1$.
Thus, since perturbations of $O(1/\sqrt{n})$ (which is tiny for large
$n$) cause the second eigenvalue to jump from 0 to 1 (and the spectral
gap from 1 to 0), it would not be possible to make tiny changes to
random matrices to arrive at a construction satisfying the required
properties in Theorem \ref{thm:constr_main}.

\section{\label{sec:lower-bound-on-phi-gen}Lower bound on the edge expansion
$\phi$ in terms of the spectral gap}

In this section, we prove the lower bound on $\phi$ in Theorem \ref{thm:gen-phi-lambda2},
and the lower bound on $\phi$ in Theorem \ref{thm:main} will follow
as a special case. The proof is a result of a sequence of lemmas that
we state next. The first lemma states that $\phi$ is sub-multiplicative
in the following sense. 
\begin{lem}
\label{lem:pow_bound} Let $R\in\mathbb{R}^{n\times n}$ be a nonnegative
matrix with left and right eigenvectors $u$ and $v$ for the PF eigenvalue
1. Then 
\[
\phi(R^{k})\leq k\cdot\phi(R).
\]
\end{lem}

\begin{proof}
The proof is given in Appendix \ref{sec:Proof-of-pow-bound}. 
\end{proof}
For the case of symmetric doubly stochastic matrices $R$, Lemma \ref{lem:pow_bound}
follows from a theorem of Blakley and Roy \cite{BlakleyR65}. (It
does not fall into the framework of an extension of that result to
the nonsymmetric case \cite{Pate12}). Lemma \ref{lem:pow_bound}
helps to lower bound $\phi(R)$ by taking powers of $R$, which is
useful since we can take sufficient powers in order to make the matrix
simple enough that its edge expansion is easily calculated. The next
two lemmas follow by technical calculations. 
\begin{lem}
\label{lem:T_norm_bound} Let $T\in\mathbb{C}^{n\times n}$ be an
upper triangular matrix with $\|T\|{}_{2}=\sigma$ and for every $i$,
$|T_{i,i}|\leq\beta$. Then 
\[
\|T^{k}\|_{2}\leq n\cdot\sigma^{n}\cdot{k+n \choose n}\cdot\beta^{k-n}.
\]
\end{lem}

\begin{proof}
The proof is given in Appendix \ref{sec:Proof-of-T-norm-bound}. 
\end{proof}
Using Lemma \ref{lem:T_norm_bound}, we can show the following lemma
for the special case of upper triangular matrices with operator norm
at most 1. 
\begin{lem}
\label{lem:T_norm_1_bound} Let $T\in\mathbb{C}^{n\times n}$ be an
upper triangular matrix with $\|T\|_{2}\leq1$ and $|T_{i,i}|\leq\alpha<1$
for every $i$. Then $\|T^{k}\|\leq\epsilon$ for 
\[
k\geq\dfrac{4n+2\ln(\frac{n}{\epsilon})}{1-\alpha}.
\]
\end{lem}

\begin{proof}
The proof is given in Appendix \ref{sec:Proof-of-T-norm-1-bound}. 
\end{proof}
Given lemmas \ref{lem:pow_bound} and \ref{lem:T_norm_1_bound}, we
can lower bound $\phi(R)$ in terms of $1-|\lambda_{m}(R)|$ (where
$\lambda_{m}$ is the nontrivial eigenvalue that is maximum in magnitude).
Our aim is to lower bound $\phi(R)$ by $\phi(R^{k})$, but since
the norm of $R^{k}$ increases by powering, we cannot use the lemmas
directly, since we do not want a dependence on $\sigma(R)$ in the
final bound. To handle this, we transform $R$ to $A$, such that
$\phi(R)=\phi(A)$, the eigenvalues of $R$ and $A$ are the same,
but $\sigma(A)=\|A\|_{2}=1$ irrespective of the norm of $R$. 
\begin{lem}
\label{lem:transRtoA} Let $R$ be a nonnegative matrix with positive
(left and right) eigenvectors $u$ and $v$ for the PF eigenvalue
1, normalized so that $\langle u,v\rangle=1$. Define $A=D_{u}^{\frac{1}{2}}D_{v}^{-\frac{1}{2}}RD_{u}^{-\frac{1}{2}}D_{v}^{\frac{1}{2}}$.
Then the following hold for $A$: 
\begin{enumerate}
\item $\phi(A)=\phi(R)$. 
\item For every $i$, $\lambda_{i}(A)=\lambda_{i}(R)$. 
\item $\|A\|_{2}=1$. 
\end{enumerate}
\end{lem}

\begin{proof}
The proof is given in Appendix \ref{sec:Proof-of-transRtoA}. 
\end{proof}
Given Lemma \ref{lem:transRtoA}, we lower bound $\phi(A)$ using
$\phi(A^{k})$ in terms of $1-|\lambda_{m}(A)|$, to obtain the corresponding
bounds for $R$. 
\begin{lem}
\label{lem:mod_lambda_bound}Let $R$ be a nonnegative matrix with
positive (left and right) eigenvectors $u$ and $v$ for the PF eigenvalue
1, normalized so that $\langle u,v\rangle=1$. Let $\lambda_{m}$
be the nontrivial eigenvalue of $R$ that is maximum in magnitude
and let $\kappa=\min_{i}u_{i}\cdot v_{i}$. Then 
\[
\frac{1}{20}\cdot\frac{1-|\lambda_{m}|}{n+\ln\left(\dfrac{1}{\kappa}\right)}\leq\phi(R).
\]
\end{lem}

\begin{proof}
The proof is given in Appendix \ref{sec:Proof-of-mod-lambda-bound}. 
\end{proof}
Given Lemma \ref{lem:mod_lambda_bound}, we use the trick of lazy
random walks to get a bound on $1-\re\lambda_{2}(R)$ from a bound
on $1-|\lambda_{m}(R)|$.
\begin{lem}
\label{lem:real-lambda-bound} Let $R$ be a nonnegative matrix with
positive (left and right) eigenvectors $u$ and $v$ for the PF eigenvalue
1, normalized so that $\langle u,v\rangle=1$. Let $\kappa=\min_{i}u_{i}\cdot v_{i}$.
Then 
\[
\frac{1}{30}\cdot\dfrac{1-\re\lambda_{2}(R)}{n+\ln\left(\frac{1}{\kappa}\right)}\leq\phi(R).
\]
For any doubly stochastic matrix $A$, 
\[
\dfrac{1-\re\lambda_{2}(A)}{35\cdot n}\leq\phi(A),
\]
and thus 
\[
\dfrac{1}{35\cdot n}\leq\Gamma(n).
\]
\end{lem}

\begin{proof}
The proof is given in Appendix \ref{sec:Proof-of-real-lambda-bound}. 
\end{proof}
This completes the proof of the lower bound on $\phi$ in Theorem
\ref{thm:gen-phi-lambda2}, and the upper bound on $\phi$ in Theorem
\ref{thm:gen-phi-lambda2} is shown in Appendix \ref{sec:Proof-of-gen-fied}.
Combined with Theorem \ref{thm:constr_main}, this also completes
the proof of Theorem \ref{thm:main}.

\section{\label{sec:Mix}Mixing time}

We now study the mixing time of nonnegative matrices, and relate it
to all the quantities we have studied so far. To motivate the definition
of mixing time for general nonnegative matrices, we first consider
the mixing time of doubly stochastic matrices. The mixing time of
a doubly stochastic matrix $A$ (i.e., of the underlying Markov chain)
is the worst-case number of steps required for a random walk starting
at any vertex to reach a distribution approximately uniform over the
vertices. To avoid complications of periodic chains, we assume that
$A$ is $\frac{1}{2}$-lazy, meaning that for every $i$, $A_{i,i}\geq\frac{1}{2}$.
Given any doubly stochastic matrix $A$, it can be easily converted
to the lazy random walk $\frac{1}{2}I+\frac{1}{2}A$. This is still
doubly stochastic and in the conversion both $\phi(A)$ and the spectral
gap are halved. The mixing time will be finite provided only that
the chain is connected. Consider the indicator vector $\mathbf{1}_{\{i\}}$
for any vertex $i$. We want to find the smallest $\tau$ such that
$A^{\tau}\mathbf{1}_{\{i\}}\approx\frac{1}{n}\mathbf{1}$ or $A^{\tau}\mathbf{1}_{\{i\}}-\frac{1}{n}\mathbf{1}\approx0$,
which can further be written as $\left(A^{\tau}-\frac{1}{n}\mathbf{1}\cdot\mathbf{1}^{T}\right)\mathbf{1}_{\{i\}}\approx0$.
Concretely, for any $\epsilon$, we want to find $\tau=\tau_{\epsilon}(A)$
such that for any $i$, 
\[
\left\Vert \left(A^{\tau}-\frac{1}{n}\mathbf{1}\cdot\mathbf{1}^{T}\right)\mathbf{1}_{\{i\}}\right\Vert _{1}\leq\epsilon.
\]
Given such a value of $\tau$, for any vector $x$ such that $\|x\|_{1}=1$,
we get 
\[
\left\Vert \left(A^{\tau}-\frac{1}{n}\mathbf{1}\cdot\mathbf{1}^{T}\right)x\right\Vert _{1}=\left\Vert \sum\limits _{i}\left(A^{\tau}-\frac{1}{n}\mathbf{1}\cdot\mathbf{1}^{T}\right)x_{i}\mathbf{1}_{\{i\}}\right\Vert _{1}\leq\sum\limits _{i}|x_{i}|\left\Vert \left(A^{\tau}-\frac{1}{n}\mathbf{1}\cdot\mathbf{1}^{T}\right)\mathbf{1}_{\{i\}}\right\Vert _{1}\leq\sum\limits _{i}|x_{i}|\cdot\epsilon=\epsilon.
\]
Thus, the mixing time $\tau_{\epsilon}(A)$ is the number $\tau$
for which $\left\Vert (A^{\tau}-J)\cdot x\right\Vert _{1}\leq\epsilon$
for any $x$ such that $\|x\|_{1}=1$.

We want to extend this definition to any nonnegative matrix $R$ with
PF eigenvalue 1 and corresponding positive left and right eigenvectors
$u$ and $v$. Note that if $R$ is reducible (i.e., $\phi(R)=0$),
then the mixing time is infinite. Further, if $R$ is periodic, then
mixing time is again ill-defined. Thus, we again assume that $R$
is irreducible and $\frac{1}{2}$-lazy, i.e. $R_{i,i}\geq\frac{1}{2}$
for every $i$. Let $x$ be any nonnegative vector for the sake of
exposition, although our final definition will not require nonnegativity
and will hold for any $x$. We want to find $\tau$ such that $R^{\tau}x$
about the same as the component of $x$ along the direction of $v$.
Further, since we are right-multiplying and want convergence to the
right eigenvector $v$, we will define the $\ell_{1}$-norm using
the left eigenvector $u$. Thus, for the starting vector $x$, instead
of requiring $\|x\|_{1}=1$ as in the doubly stochastic case, we will
require $\|D_{u}x\|_{1}=1$. Since $x$ is nonnegative, $\|D_{u}x\|_{1}=\langle u,x\rangle=1$.
Thus, we want to find $\tau$ such that $R^{\tau}x\approx v$, or
$\left(R^{\tau}-v\cdot u^{T}\right)x\approx0$. Since we measured
the norm of the starting vector $x$ with respect to $u$, we will
also measure the norm of the final vector $\left(R^{\tau}-v\cdot u^{T}\right)x$
with respect to $u$. Thus we arrive at the following definition. 
\begin{defn}
\label{def:gen-Mixing-time}(Mixing time of general nonnegative matrices
$R$) Let $R$ be a $\frac{1}{2}$-lazy, irreducible nonnegative matrix
with PF eigenvalue 1 with $u$ and $v$ as the corresponding positive
left and right eigenvectors, where $u$ and $v$ are normalized so
that $\langle u,v\rangle=\|D_{u}v\|_{1}=1$. Then the mixing time
$\tau_{\epsilon}(R)$ is the smallest number $\tau$ such that $\left\Vert D_{u}\left(R^{\tau}-v\cdot u^{T}\right)x\right\Vert _{1}\leq\epsilon$
for every vector $x$ with $\|D_{u}x\|_{1}=1$. 
\end{defn}

We remark that similar to the doubly stochastic case, using the triangle
inequality, it is sufficient to find mixing time of standard basis
vectors $\mathbf{1}_{\{i\}}$. Let $y_{i}=\frac{\mathbf{1}_{\{i\}}}{\|D_{u}\mathbf{1}_{\{i\}}\|_{1}}$,
then $y_{i}$ is nonnegative, $\|D_{u}y_{i}\|_{1}=\langle u,y_{i}\rangle=1$,
then for any $x$, such that $\|D_{u}x\|_{1}=1$, we can write 
\[
x=\sum\limits _{i}c_{i}\mathbf{1}_{\{i\}}=\sum_{i}c_{i}\|D_{u}\mathbf{1}_{\{i\}}\|_{1}y_{i}
\]
with 
\[
\|D_{u}x\|_{1}=\left\Vert D_{u}\sum\limits _{i}c_{i}\mathbf{1}_{\{i\}}\right\Vert _{1}=\sum\limits _{i}|c_{i}|\|D_{u}\mathbf{1}_{\{i\}}\|_{1}=1.
\]
Thus, if for every $i$, $\left\Vert D_{u}\left(R^{\tau}-v\cdot u^{T}\right)y_{i}\right\Vert _{1}\leq\epsilon$,
then 
\[
\left\Vert D_{u}\left(R^{\tau}-v\cdot u^{T}\right)x\right\Vert _{1}=\left\Vert D_{u}\left(R^{\tau}-v\cdot u^{T}\right)\sum\limits _{i}c_{i}\|D_{u}\mathbf{1}_{\{i\}}\|_{1}y_{i}\right\Vert _{1}\leq\sum\limits _{i}c_{i}\|D_{u}\mathbf{1}_{\{i\}}\|_{1}\left\Vert D_{u}\left(R^{\tau}-v\cdot u^{T}\right)y_{i}\right\Vert _{1}\leq\epsilon.
\]
Thus, it is sufficient to find mixing time for every nonnegative $x$
with $\|D_{u}x\|_{1}=\langle u,x\rangle=1$, and it will hold for
all $x$.

For the case of \emph{reversible }nonnegative matrices $M$ with PF
eigenvalue 1, the mixing time is well-understood, and it is easily
shown that 
\begin{equation}
\tau_{\epsilon}(M)\leq\dfrac{\ln\left(\frac{n}{\kappa\cdot\epsilon}\right)}{1-\lambda_{2}(M)}\stackrel{\text{(Theorem \ref{thm:cheeg})}}{\leq}\dfrac{2\cdot\ln\left(\frac{n}{\kappa\cdot\epsilon}\right)}{\phi^{2}(M)}.\label{eq:sym-mixt}
\end{equation}
We will give corresponding bounds for the mixing time of general nonnegative
matrices.

\subsection{Mixing time and singular values}

We first show a simple lemma relating the mixing time of nonnegative
matrices to the second singular value. This lemma is powerful enough
to recover the bounds obtained by Fill \cite{fill1991eigenvalue}
and Mihail \cite{mihail1989conductance} in an elementary way. Since
the largest singular value of any general nonnegative matrix $R$
with PF eigenvalue 1 could be much larger than 1, the relation between
mixing time and second singular value makes sense only for nonnegative
matrices with the same left and right eigenvector for eigenvalue 1,
which have largest singular value 1 by Lemma \ref{lem:transRtoA}. 
\begin{lem}
\label{lem:mixt-singval} (Mixing time and second singular value)
Let $A$ be a nonnegative matrix (not necessarily lazy) with PF eigenvalue
1, such that $Aw=w$ and $A^{T}w=w$ for some $w$ with $\langle w,w\rangle=1$,
and let $\kappa=\min_{i}w_{i}^{2}$. Then for every $c>0$, 
\[
\tau_{\epsilon}(A)\ \leq\ \dfrac{c\cdot\ln\left(\frac{\sqrt{n}}{\sqrt{\kappa}\cdot\epsilon}\right)}{1-\sigma_{2}^{c}(A)}\ \leq\ \dfrac{c\cdot\ln\left(\frac{n}{\kappa\cdot\epsilon}\right)}{1-\sigma_{2}^{c}(A)}.
\]
\end{lem}

\begin{proof}
The proof is given in Appendix \ref{sec:Proof-of-mixt-sing-bound}. 
\end{proof}
For the case of $c=2$, Lemma \ref{lem:mixt-singval} was obtained
by Fill \cite{fill1991eigenvalue}, but we find our proof simpler.

\subsection{Mixing time and edge expansion}

We now relate the mixing time of general nonnegative matrices $R$
to its edge expansion $\phi(R)$. The upper bound for row stochastic
matrices $R$ in terms of $\phi(R)$ were obtained by Mihail \cite{mihail1989conductance}
and simplified by Fill \cite{fill1991eigenvalue} using Lemma \ref{lem:mixt-singval}
for $c=2$. Thus, the following lemma is not new, but we prove it
in Appendix \ref{sec:Proof-of-mixt-phi-rel} for completeness, since
our proof is simpler and holds for any nonnegative matrix $R$. 
\begin{lem}
\label{lem:mixt-phi-rel}(Mixing time and edge expansion) Let $\tau_{\epsilon}(R)$
be the mixing time of a $\frac{1}{2}$-lazy nonnegative matrix $R$
with PF eigenvalue 1 and corresponding positive left and right eigenvectors
$u$ and $v$, and let $\kappa=\min_{i}u_{i}\cdot v_{i}$. Then 
\begin{align*}
\frac{\frac{1}{2}-\epsilon}{\phi(R)}\ \leq\ \tau_{\epsilon}(R)\  & \leq\ \dfrac{4\cdot\ln\left(\frac{n}{\kappa\cdot\epsilon}\right)}{\phi^{2}(R)}.
\end{align*}
\end{lem}

\begin{proof}
The proof is given in Appendix \ref{sec:Proof-of-mixt-phi-rel}. 
\end{proof}

\subsection{Mixing time and spectral gap}

We obtain bounds for the mixing time of nonnegative matrices in terms
of the spectral gap, using methods similar to the ones used to obtain
the upper bound on $\phi$ in Theorem \ref{thm:gen-phi-lambda2}. 
\begin{lem}
\label{lem:mixt-lambda2} (Mixing time and spectral gap) Let $\tau_{\epsilon}(R)$
be the mixing time of a $\frac{1}{2}$-lazy nonnegative matrix $R$
with PF eigenvalue 1 and corresponding positive left and right eigenvectors
$u$ and $v$, and let $\kappa=\min_{i}u_{i}\cdot v_{i}$. Then 
\begin{align*}
\frac{\frac{1}{2}-\epsilon}{\sqrt{2\cdot\left(1-\re\lambda_{2}(R)\right)}}\ \leq\ \tau_{\epsilon}(R)\  & \leq\ 20\cdot\frac{n+\ln\left(\dfrac{1}{\kappa\cdot\epsilon}\right)}{1-\re\lambda_{2}(R)}.
\end{align*}
\end{lem}

\begin{proof}
The proof is given in Appendix \ref{sec:Proof-of-mixt-lambda2} 
\end{proof}
We remark that there is only \emph{additive} and not multiplicative
dependence on $\ln\left(\frac{n}{\kappa\cdot\epsilon}\right)$. Further,
our construction for the upper bound in Theorem \ref{thm:main} also
shows that the upper bound on $\tau$ using $\re\lambda_{2}$ in Lemma
\ref{lem:mixt-lambda2} is also (almost) tight. For the construction
of $A_{n}$ in Theorem \ref{thm:main}, letting the columns of $U_{n}$
be $u_{1},\ldots,u_{n}$, for $x=u_{2}$, $(A_{n}^{k}-J)u_{2}=(1-(2+\sqrt{n})^{-1})^{k}u_{3}$,
and so for $k=O(\sqrt{n})$, the triangular block of $A^{O(\sqrt{n})}$
has norm about $1/e$, which further becomes less than $\epsilon$
after about $\ln\left(\frac{n}{\epsilon}\right)$ powers. Thus for
the matrices $A_{n}$, $\tau_{\epsilon}(A_{n})\in O\left(\sqrt{n}\cdot\ln\left(\frac{n}{\epsilon}\right)\right)$.
This shows Lemma \ref{lem:mixt-lambda2} is also (almost) tight since
$\lambda_{2}(A_{n})=0$.

\subsection{Mixing time of a nonnegative matrix and its additive symmetrization}

We can also bound the mixing time of a nonnegative matrix $A$ with
the same left and right eigenvector $w$ for PF eigenvalue 1, with
the mixing time of its \emph{additive symmetrization} $M=\frac{1}{2}(A+A^{T})$.
Note that we obtained a similar bound on the spectral gaps of $A$
and $M$ in Lemma \ref{lem:bound-eig-A-and_AplusAT}. Since $\phi(A)=\phi(M)$,
we can bound $\tau_{\epsilon}(A)$ and $\tau_{\epsilon}(M)$ using
the two sided bounds between edge expansion and mixing time in Lemma
\ref{lem:mixt-phi-rel}. For the lower bound, we get $\gamma_{1}\cdot\sqrt{\tau_{\epsilon}(M)}\leq\tau_{\epsilon}(A)$,
and for the upper bound, we get 
\[
\tau_{\epsilon}(A)\leq\gamma_{2}\cdot\tau_{\epsilon}^{2}(M),
\]
where $\gamma_{1}$ and $\gamma_{2}$ are some functions polylogarithmic
in $n,\kappa,\frac{1}{\epsilon}$. However, by bounding the appropriate
operator, we can show a tighter upper bound on $\tau_{\epsilon}(A)$,
with only a \emph{linear} instead of quadratic dependence on $\tau_{\epsilon}(M)$. 
\begin{lem}
\label{lem:tau-A-AAT}Let $A$ be a $\frac{1}{2}$-lazy nonnegative
matrix with positive left and right eigenvector $w$ for PF eigenvalue
1, let $M=\frac{1}{2}(A+A^{T})$, and $\kappa=\min_{i}w_{i}^{2}$.
Then 
\[
\frac{1-2\epsilon}{4\cdot\ln^{\frac{1}{2}}\left(\frac{n}{\kappa\cdot\epsilon}\right)}\cdot\tau_{\epsilon}^{\frac{1}{2}}\left(M\right)\ \leq\ \tau_{\epsilon}(A)\ \leq\ \frac{2\cdot\ln\left(\frac{n}{\kappa\cdot\epsilon}\right)}{\ln\left(\frac{1}{\epsilon}\right)}\cdot\tau_{\epsilon}(M).
\]
\end{lem}

\begin{proof}
The proof is given in Appendix \ref{sec:Proof-of-mixt-AAT}. 
\end{proof}
One example application of Lemma \ref{lem:tau-A-AAT} is the following:
given any undirected graph $G$ such that each vertex has degree $d$,
\emph{any} manner of orienting the edges of $G$ to obtain a graph
in which every vertex has in-degree and out-degree $d/2$ cannot increase
the mixing time of a random walk (up to a factor of $\ln\left(\frac{n}{\kappa\cdot\epsilon}\right)$).

\subsection{Mixing time of the continuous operator}

Let $R$ be a nonnegative matrix with PF eigenvalue 1 and associated
positive left and right eigenvectors $u$ and $v$. The continuous
time operator associated with $R$ is defined as $\exp\left(t\cdot\left(R-I\right)\right)$,
where for any matrix $M$, we formally define $\exp(M)=\sum_{i=0}\frac{1}{i!}M^{i}$.
The reason this operator is considered continuous, is that starting
with any vector $x_{0}$, the vector $x_{t}$ at time $t\in\mathbb{R}_{\geq0}$
is defined as $x_{t}=\exp\left(t\cdot\left(R-I\right)\right)x_{0}$.
Since 
\[
\exp\left(t\cdot\left(R-I\right)\right)=\exp(t\cdot R)\cdot\exp(-t\cdot I)=e^{-t}\sum_{i=0}^{\infty}\frac{1}{i!}t^{i}R^{i}
\]
where we split the operator into two terms since $R$ and $I$ commute,
it follows that $\exp\left(t\cdot\left(R-I\right)\right)$ is nonnegative,
and if $\lambda$ is any eigenvalue of $R$ for eigenvector $y$,
then $e^{t(\lambda-1)}$ is an eigenvalue of $\exp\left(t\cdot\left(R-I\right)\right)$
for the same eigenvector $y$. Thus, it further follows that $u$
and $v$ are the left and right eigenvectors for $\exp\left(t\cdot\left(R-I\right)\right)$
with PF eigenvalue 1. The mixing time of $\exp\left(t\cdot\left(R-I\right)\right)$,
is the value of $t$ for which 
\[
\left\Vert D_{u}\left(\exp\left(t\cdot\left(R-I\right)\right)-v\cdot u^{T}\right)v_{0}\right\Vert _{1}\leq\epsilon
\]
for every $v_{0}$ such that $\|D_{u}v_{0}\|_{1}=1$, and thus, it
is exactly same as considering the mixing time of $\exp(R-I)$ in
the sense of Definition \ref{def:gen-Mixing-time}. 
\begin{lem}
\label{lem:mixt-contchain}Let $R$ be a nonnegative matrix (not necessarily
lazy) with positive left and right eigenvectors $u$ and $v$ for
PF eigenvalue 1, normalized so that $\langle u,v\rangle=1$ and let
$\kappa=\min_{i}u_{i}\cdot v_{i}$. Then the mixing time of $\exp(t\cdot(R-I))$,
or $\tau_{\epsilon}\left(\exp(R-I)\right)$ is bounded as 
\[
\frac{\frac{1}{2}-\epsilon}{\phi(R)}\ \leq\ \tau_{\epsilon}\left(\exp(R-I)\right)\ \leq\ \frac{100\cdot\ln\left(\frac{n}{\kappa\cdot\epsilon}\right)}{\phi^{2}(R)}.
\]
\end{lem}

\begin{proof}
The proof is given in Appendix \ref{sec:Proof-of-mixt-contchain}. 
\end{proof}

\subsection{Bounds using the canonical paths method}

For the case of symmetric nonnegative matrices $M$ with PF eigenvalue
1, as stated earlier in this section in equation \ref{eq:sym-mixt},
since $\tau$ varies inversely with $1-\lambda_{2}$ (up to a loss
of a factor of $\ln(\frac{n}{\kappa\cdot\epsilon})$), it follows
that any lower bound on the spectral gap can be used to upper bound
$\tau_{\epsilon}(M)$. Further, since $1-\lambda_{2}$ can be written
as a minimization problem for symmetric matrices (see Section \ref{sec:Preliminaries}),
any relaxation of the optimization problem can be used to obtain a
lower bound on $1-\lambda_{2}$, and inequalities obtained thus are
referred to as \emph{Poincare inequalities}. One such method is to
use \emph{canonical paths }\cite{sinclair1992improved} in the underlying
weighted graph, which helps to bound mixing time in certain cases
in which computing $\lambda_{2}$ or $\phi$ is infeasible. However,
since it is possible to define canonical paths in many different ways,
it leads to multiple relaxations to bound $1-\lambda_{2}$, each useful
in a different context. We remark one particular definition and lemma
here, since it is relevant to our construction in Theorem \ref{thm:constr_main},
after suitably modifying it for the doubly stochastic case. 
\begin{lem}
\cite{sinclair1992improved}\label{lem:canpath-sinc} Let $M$ represent
a symmetric doubly stochastic matrix. Let $W$ be a set of paths in
$M$, one between every pair of vertices. For any path $\gamma_{u,v}\in S$
between vertices $(u,v)$ where $\gamma_{u,v}$ is simply a set of
edges between $u$ and $v$, let the number of edges or the (unweighted)
length of the path be $|\gamma_{u,v}|$. Let 
\[
\rho_{W}(M)=\max_{e=(x,y)}\dfrac{\sum\limits _{(u,v):e\in\gamma_{u,v}}|\gamma_{u,v}|}{n\cdot M_{x,y}}.
\]
Then for any $W$, 
\[
1-\lambda_{2}(M)\geq\frac{1}{\rho_{W}(M)}
\]
and thus, 
\[
\tau_{\epsilon}(M)\leq\rho_{W}(M)\cdot\ln\left(\frac{n}{\epsilon}\right).
\]
\end{lem}

\begin{cor}
\label{cor:canpath-sinc}Combining Lemma \ref{lem:sigma-bound-phi}
and Lemma \ref{lem:canpath-sinc}, it follows that for any doubly
stochastic matrix $A$, and any set $W$ of paths in the underlying
graph of $AA^{T}$, 
\[
\tau_{\epsilon}(A)\leq\dfrac{2\cdot\ln\left(\frac{n}{\epsilon}\right)}{1-\sigma_{2}^{2}(A)}=\dfrac{2\cdot\ln\left(\frac{n}{\epsilon}\right)}{1-\lambda_{2}(AA^{T})}\leq2\cdot\rho_{W}(AA^{T})\cdot\ln\left(\frac{n}{\epsilon}\right).
\]
\end{cor}

Consider the example $A_{n}$ in Theorem \ref{thm:constr_main}. It
is not difficult to see that 
\begin{equation}
\tau_{\epsilon}(A_{n})\in O\left(\sqrt{n}\cdot\ln\left(\frac{n}{\epsilon}\right)\right).\label{eq:para_mixt}
\end{equation}
This follows since the factor of $\sqrt{n}$ ensures that the only
non zero entries in the triangular matrix $T_{n}$ (see Appendix \ref{sec:Proof-of-construction})
in the Schur form of $A^{\lceil\sqrt{n}\rceil}$ are about $e^{-1}$,
and the factor of $\ln\left(\frac{n}{\epsilon}\right)$ further converts
these entries to have magnitude at most $\frac{\epsilon}{n}$ in $A^{\tau}$.
Thus, the operator norm becomes about $\frac{\epsilon}{n}$, and the
$\ell_{1}$ norm gets upper bounded by $\epsilon$. However, from
Theorem \ref{thm:constr_main}, since $\phi(A_{n})\geq\frac{1}{6\sqrt{n}},$
it follows from Lemma \ref{lem:mixt-phi-rel} that $\tau_{\epsilon}(A_{n})\in O\left(n\cdot\ln\left(\frac{n}{\epsilon}\right)\right)$,
about a quadratic factor off from the actual upper bound in equation
\ref{eq:para_mixt}. Further, from Theorem \ref{thm:constr_main},
the second eigenvalue of $A_{n}$ is $0$, and even employing Lemma
\ref{lem:mixt-lambda2} leads to a quadratic factor loss from the
actual bound. However, Lemma \ref{lem:mixt-singval} and Corollary
\ref{cor:canpath-sinc} do give correct bounds. Since $\sigma_{2}(A_{n})=1-\frac{1}{\sqrt{n}+2}$
from Theorem \ref{thm:constr_main} (see the proof in Appendix \ref{sec:Proof-of-construction}),
it follows from Lemma \ref{lem:mixt-singval} for $c=1$ that $\tau_{\epsilon}(A_{n})\in O\left(\sqrt{n}\cdot\ln\left(\frac{n}{\epsilon}\right)\right)$,
matching the bound in equation \ref{eq:para_mixt}. Now to see the
bound given by canonical paths and corollary \ref{cor:canpath-sinc},
consider the matrix $M=A_{n}A_{n}^{T}$. Every entry of $M$ turns
out to be positive, and the set $W$ is thus chosen so that the path
between any pair of vertices is simply the edge between the vertices.
Further for $r_{n},\alpha_{n},\beta_{n}$ defined in the proof (Appendix
\ref{sec:Proof-of-construction}) of Theorem \ref{thm:constr_main},
$M=J+r_{n}^{2}B$, where 
\[
B_{1,1}=\frac{n-2}{n},\ \ B_{n,n}=(n-2)\cdot\beta_{n}^{2},\ \ B_{i,i}=\alpha_{n}^{2}+(n-3)\cdot\beta_{n}^{2},\ \ B_{1,n}=B_{n,1}=\frac{n-2}{\sqrt{n}}\cdot\beta_{n},
\]
\[
B_{n,j}=B_{j,n}=\alpha_{n}\cdot\beta_{n}+(n-3)\cdot\beta_{n}^{2},\ \ B_{1,j}=B_{j,1}=\frac{1}{\sqrt{n}}\cdot(\alpha_{n}+(n-3)\cdot\beta_{n}),\ \ B_{i,j}=2\cdot\alpha_{n}\cdot\beta_{n}+(n-4)\cdot\beta_{n}^{2},
\]
and $2\leq i,j\leq n-1$. It follows that any entry of the matrix
$M$ is at least $c\cdot n^{-\frac{3}{2}}$ (for some constant $c$),
and from Corollary \ref{cor:canpath-sinc}, we get that $\tau_{\epsilon}(A_{n})\in O\left(\sqrt{n}\cdot\ln\left(\frac{n}{\epsilon}\right)\right)$,
matching the bound in equation \ref{eq:para_mixt}.

\section*{Acknowledgements}

We thank Shanghua Teng, Ori Parzanchevski and Umesh Vazirani for illuminating
conversations. JCM would like to thank Luca Trevisan for his course
``Graph Partitioning and Expanders'' on VentureLabs, which got the
author interested in Spectral Graph Theory. Research was supported
by NSF grants 1319745, 
1553477, 
1618795, and 1909972; 
BSF grant 2012333; and, during a residency of LJS at the Israel Institute
for Advanced Studies, by a EURIAS Senior Fellowship co-funded by the
Marie Sk{\l }odowska-Curie Actions under the 7th Framework Programme.

\pagebreak{}

 \bibliographystyle{alpha}
\bibliography{refs}

\begin{thebibliography}{DSC96}

\bibitem[Alo86]{alon86}
N.~Alon.
\newblock Eigenvalues and expanders.
\newblock {\em Combinatorica}, 6(2):83--96, 1986.

\bibitem[AM85]{AlonM85}
N.~Alon and V.~Milman.
\newblock $\lambda_1$, isoperimetric inequalities for graphs, and
  superconcentrators.
\newblock {\em Journal of Combinatorial Theory, Series B}, 38:73--88, 1985.

\bibitem[BF60]{bauer1960norms}
F.~L. Bauer and C.~T. Fike.
\newblock Norms and exclusion theorems.
\newblock {\em Numerische Mathematik}, 2(1):137--141, 1960.

\bibitem[Bha97]{Bhatia97}
R.~Bhatia.
\newblock {\em Matrix Analysis}.
\newblock Springer, 1997.

\bibitem[BR65]{BlakleyR65}
G.~R. Blakley and P.~Roy.
\newblock A {H}\"older type inequality for symmetric matrices with nonnegative
  entries.
\newblock {\em Proc. American Mathematical Society}, 6(16):1244--1245, 1965.

\bibitem[Bru46]{de1946combinatorial}
N.~G.~de Bruijn.
\newblock A combinatorial problem.
\newblock {\em Proceedings of the Section of Sciences of the Koninklijke
  Nederlandse Akademie van Wetenschappen te Amsterdam}, 49(7):758--764, 1946.

\bibitem[Bus82]{Buser82}
P.~Buser.
\newblock A note on the isoperimetric constant.
\newblock {\em Ann. Sci. \'Ecole Norm. Sup.}, 15(2):213--230, 1982.

\bibitem[Che70]{cheeger70}
J.~Cheeger.
\newblock A lower bound for the smallest eigenvalue of the {L}aplacian.
\newblock In R.~C. Gunning, editor, {\em Problems in analysis, a symposium in
  honor of S.\ Bochner}, pages 195--199. Princeton Univ.\ Press, 1970.

\bibitem[Dod84]{dodziuk84}
J.~Dodziuk.
\newblock Difference equations, isoperimetric inequality and transience of
  certain random walks.
\newblock {\em Trans. Amer. Math. Soc.}, 284(2):787--794, 1984.

\bibitem[DSC96]{diaconis1996logarithmic}
P.~Diaconis and L.~Saloff-Coste.
\newblock Logarithmic {S}obolev inequalities for finite {M}arkov chains.
\newblock {\em Annals of Applied Probability}, 6(3):695--750, 1996.

\bibitem[DT98]{delorme1998spectrum}
C.~Delorme and J.-P. Tillich.
\newblock The spectrum of de {B}ruijn and {K}autz graphs.
\newblock {\em European Journal of Combinatorics}, 19(3):307--319, 1998.

\bibitem[Fie95]{fiedler1995estimate}
M.~Fiedler.
\newblock An estimate for the nonstochastic eigenvalues of doubly stochastic
  matrices.
\newblock {\em Linear algebra and its applications}, 214:133--143, 1995.

\bibitem[Fil91]{fill1991eigenvalue}
J.~A. Fill.
\newblock Eigenvalue bounds on convergence to stationarity for nonreversible
  {M}arkov chains, with an application to the exclusion process.
\newblock {\em Annals of Applied Probability}, 1(1):62--87, 1991.

\bibitem[Fro12]{frobenius1912matrizen}
G.~Frobenius.
\newblock {\"U}ber matrizen aus nicht negativen elementen.
\newblock 1912.

\bibitem[GJS74]{garey1974some}
M.~R. Garey, D.~S. Johnson, and L.~Stockmeyer.
\newblock Some simplified {NP}-complete problems.
\newblock In {\em Proceedings of the sixth annual ACM symposium on Theory of
  computing}, pages 47--63. ACM, 1974.

\bibitem[Kla84]{klawe1984limitations}
M.~Klawe.
\newblock Limitations on explicit constructions of expanding graphs.
\newblock {\em SIAM Journal on Computing}, 13(1):156--166, 1984.

\bibitem[Lax07]{lax07}
P.~D. Lax.
\newblock {\em Linear Algebra and Its Applications}.
\newblock Wiley-Interscience, Hoboken, NJ, second edition, 2007.

\bibitem[Li92]{Li92}
W.-C.~W. Li.
\newblock Character sums and {A}belian {R}amanujan graphs.
\newblock {\em J. Number Theory}, 41:199--217, 1992.
\newblock Appendix by K. Feng and W.-C. W. Li.

\bibitem[Mih89]{mihail1989conductance}
M.~Mihail.
\newblock Conductance and convergence of {M}arkov chains - a combinatorial
  treatment of expanders.
\newblock In {\em 30th Annual symposium on Foundations of computer science},
  pages 526--531. IEEE, 1989.

\bibitem[MT06]{MonT06}
R.~Montenegro and P.~Tetali.
\newblock Mathematical aspects of mixing times in {M}arkov chains.
\newblock {\em Foundations and Trends in Theoretical Computer Science},
  1(3):237--354, 2006.

\bibitem[Nil91]{Nilli91}
A.~Nilli.
\newblock On the second eigenvalue of a graph.
\newblock {\em Discrete Math.}, 91:207--210, 1991.

\bibitem[Pat12]{Pate12}
T.~H. Pate.
\newblock Extending the {H}\"older type inequality of {B}lakley and {R}oy to
  non-symmetric non-square matrices.
\newblock {\em Trans. American Mathematical Society}, 364(8):4267--4281, 2012.

\bibitem[Per07]{perron1907theorie}
O.~Perron.
\newblock Zur theorie der matrices.
\newblock {\em Mathematische Annalen}, 64(2):248--263, 1907.

\bibitem[Saa11]{Saad11}
Y.~Saad.
\newblock {\em Numerical methods for large eigenvalue problems}.
\newblock SIAM, second edition, 2011.

\bibitem[Sin92]{sinclair1992improved}
A.~Sinclair.
\newblock Improved bounds for mixing rates of {M}arkov chains and
  multicommodity flow.
\newblock {\em Combinatorics, probability and Computing}, 1(4):351--370, 1992.

\bibitem[SJ89]{SinclairJ89}
A.~Sinclair and M.~Jerrum.
\newblock Approximate counting, uniform generation and rapidly mixing {M}arkov
  chains.
\newblock {\em Information and Computation}, 82(1):93 -- 133, 1989.

\bibitem[Vaz17]{UVaz-pers}
U.~Vazirani.
\newblock Personal communication, 2017.

\end{thebibliography}

\pagebreak{}

\appendix

\section{\label{sec:Proof-of-sigmabound}Proof of Lemma \ref{lem:sigma-bound-phi}}
\begin{proof}
Let $w$ be the left and right eigenvector of $A$ for eigenvalue
1. Then note from part (3) of Lemma \ref{lem:transRtoA}, we have
that $\|A\|_{2}=1$. We first note that since the left and right eigenvectors
for eigenvalue 1 are the same, 
\begin{equation}
\max_{v\perp w}\frac{\|Av\|_{2}}{\|v\|_{2}}\leq\sigma_{2}(A).\label{eq:sigphieq1}
\end{equation}

To see this, let $W$ be a unitary matrix with first column $w$,
then since $Aw=w$ and $A^{*}w=A^{T}w=w$, $W^{*}AW$ has two blocks,
a $1\times1$ block containing the entry 1, and an $n-1\times n-1$
block, and let the second block's singular value decomposition be
$PDQ$ with $n-1\times n-1$ unitaries $P$ and $Q$ and diagonal
$D$. Then 
\[
A=W\begin{bmatrix}1 & 0\\
0 & P
\end{bmatrix}\begin{bmatrix}1 & 0\\
0 & D
\end{bmatrix}\begin{bmatrix}1 & 0\\
0 & Q
\end{bmatrix}W^{*},
\]
giving a singular value decomposition for $A$. Thus, for any $v\perp w$
with $\|v\|_{2}=1$, $W^{*}v$ has $0$ as the first entry, and thus
\[
\|Av\|_{2}\leq\|W\|_{2}\|P\|_{2}\|D\|_{2}\|Q\|_{2}\|W^{*}\|_{2}\|v\|_{2}=\|D\|_{2}=\sigma_{2}(A).
\]

Thus we have, 
\begin{equation}
\phi(A)=\min_{S:\sum_{i\in S}w_{i}^{2}\leq\frac{1}{2}}\frac{\langle\boldsymbol{1}_{S},D_{w}AD_{w}(\boldsymbol{1}-\boldsymbol{1}_{S})\rangle}{\langle\mathbf{1}_{s},D_{w}^{2}\mathbf{1}\rangle}=1-\max_{S:\sum_{i\in S}w_{i}^{2}\leq\frac{1}{2}}\frac{\langle\boldsymbol{1}_{S},D_{w}AD_{w}\boldsymbol{1}_{S}\rangle}{\langle\mathbf{1}_{s},D_{w}^{2}\mathbf{1}\rangle}\label{eq:phi-in-proof-up}
\end{equation}
where the second equality used the fact that $AD_{w}\boldsymbol{1}=Aw=w=D_{w}\mathbf{1}$.
Let $D_{w}\boldsymbol{1}_{S}=c\cdot w+v$, where $\langle w,v\rangle=0$.
Then 
\[
c=\langle w,D_{w}\mathbf{1}_{S}\rangle=\langle\mathbf{1}_{s},D_{w}^{2}\mathbf{1}\rangle=\sum_{i\in S}w_{i}^{2}
\]
and 
\begin{align}
\|v\|_{2}^{2} & =\langle v,D_{w}\mathbf{1}_{S}-c\cdot w\rangle\nonumber \\
 & =\langle v,D_{w}\mathbf{1}_{S}\rangle\nonumber \\
 & =\langle D_{w}\mathbf{1}_{S}-c\cdot w,D_{w}\mathbf{1}_{S}\rangle\nonumber \\
 & =\langle D_{w}\mathbf{1}_{S},D_{w}\mathbf{1}_{S}\rangle-c\cdot\langle D_{w}\mathbf{1}_{S},w\rangle\nonumber \\
 & =c\cdot(1-c).\label{eq:sigphieq2}
\end{align}
Thus, 
\begin{align*}
\max_{S:\sum_{i\in S}w_{i}^{2}\leq\frac{1}{2}}\frac{\langle\boldsymbol{1}_{S},D_{w}AD_{w}\boldsymbol{1}_{S}\rangle}{\langle\mathbf{1}_{s},D_{w}^{2}\mathbf{1}\rangle} & =\frac{\langle c\cdot w+v,A(c\cdot w+v)\rangle}{c}\\
 & \ \ \ \ \text{[since \ensuremath{D_{w}} is diagonal and so \ensuremath{D_{w}^{T}=D_{w}}]}\\
 & =\frac{c^{2}+\langle v,Av\rangle}{c}\\
 & \ \ \ \ \text{[since \ensuremath{Aw=w} and \ensuremath{A^{T}w=w}]}\\
 & \leq\frac{c^{2}+\|v\|_{2}\|Av\|_{2}}{c}\\
 & \leq\frac{c^{2}+\|v\|_{2}\sigma_{2}(A)\|v\|_{2}}{c}\\
 & \ \ \ \ \text{[from equation \ref{eq:sigphieq1} since \ensuremath{v\perp w}]}\\
 & =c+\sigma_{2}(A)\cdot(1-c)\\
 & \ \ \ \ \text{[using equation \ref{eq:sigphieq2}]}\\
 & =\sigma_{2}(A)+(1-\sigma_{2}(A))\cdot c\\
 & \leq\frac{1+\sigma_{2}(A)}{2}\\
 & \ \ \ \ \text{[since \ensuremath{c=\langle\mathbf{1}_{s},D_{w}^{2}\mathbf{1}\rangle\leq\frac{1}{2}}]}
\end{align*}
which completes the proof after replacing the above upper bound in
equation \ref{eq:phi-in-proof-up}.

A simple extension (and alternate proof) is the following. Let $H=AA^{T}$,
then $H$ is a nonnegative matrix that has $w$ as the left and right
eigenvector for eigenvalue 1. For any integer $k$, we know that $H^{k}$
is nonnegative, and since $\lambda_{2}(H)=\sigma_{2}^{2}(A)$, we
get 
\[
\frac{1-\sigma_{2}^{2\cdot k}(A)}{2}=\frac{1-\lambda_{2}^{k}(H)}{2}\leq\phi(H^{k})\leq k\cdot\phi(H)\leq k\cdot(2\cdot\phi(A))
\]
where the first inequality used Cheeger's inequality (Theorem \ref{thm:gen-Cheeger's-inequalities})
and the second and third inequalities were obtained by \ref{lem:pow_bound}
and its proof in Appendix \ref{sec:Proof-of-pow-bound}. Thus, we
get that for any integer $c\geq2$, 
\[
\frac{1-\sigma_{2}^{c}(A)}{2\cdot c}\leq\phi(A).
\]
\end{proof}

\section{\label{sec:Proof-of-construction}Proof of the main construction
in Theorem \ref{thm:constr_main}}
\begin{proof}
The following calculations are easy to check, to see that $A_{n}$
is a doubly stochastic matrix: 
\begin{enumerate}
\item $a_{n}\geq0$, $b_{n}\geq0$, $c_{n}\geq0$, $d_{n}\geq0$, $e_{n}\geq0$,
$f_{n}\geq0$. 
\item $a_{n}+b_{n}=1$. 
\item $c_{n}+d_{n}+(n-3)e_{n}=1$. 
\item $b_{n}+f_{n}+(n-2)c_{n}=1$. 
\end{enumerate}
This completes the proof of (1).

$A_{n}$ is triangularized as $T_{n}$ by the unitary $U_{n}$, i.e.
\[
A_{n}=U_{n}T_{n}U_{n}^{*},
\]
with $T_{n}$ and $U_{n}$ defined as follows. Recall that $m=\sqrt{n}$.
Let 
\[
r_{n}=1-\dfrac{1}{m+2},
\]

\begin{align*}
\alpha_{n} & =\dfrac{-n^{2}+2n-\sqrt{n}}{n\cdot(n-1)}=-1+\frac{1}{m\cdot(m+1)},
\end{align*}
\begin{align*}
\beta_{n} & =\dfrac{n-\sqrt{n}}{n\cdot(n-1)}=\frac{1}{m\cdot(m+1)},
\end{align*}
\[
T_{n}=\begin{bmatrix}1 & 0 & 0 & 0 & 0 & 0 & 0 & 0 & 0\\
0 & 0 & r_{n} & 0 & 0 & 0 & 0 & \cdots & 0\\
0 & 0 & 0 & r_{n} & 0 & 0 & 0 & \cdots & 0\\
0 & 0 & 0 & 0 & r_{n} & 0 & 0 & \cdots & 0\\
0 & 0 & 0 & 0 & 0 & r_{n} & 0 & \cdots & 0\\
0 & 0 & 0 & 0 & 0 & 0 & r_{n} & \cdots & 0\\
\vdots & \vdots & \vdots & \vdots & \vdots & \vdots & \ddots & \ddots & \vdots\\
0 & 0 & 0 & 0 & 0 & 0 & \cdots & 0 & r_{n}\\
0 & 0 & 0 & 0 & 0 & 0 & 0 & 0 & 0
\end{bmatrix},
\]
and 
\[
U_{n}=\begin{bmatrix}\frac{1}{\sqrt{n}} & \frac{1}{\sqrt{n}} & \frac{1}{\sqrt{n}} & \frac{1}{\sqrt{n}} & \frac{1}{\sqrt{n}} & \frac{1}{\sqrt{n}} & \cdots & \frac{1}{\sqrt{n}} & \frac{1}{\sqrt{n}}\\
\frac{1}{\sqrt{n}} & \alpha_{n} & \beta_{n} & \beta_{n} & \beta_{n} & \beta_{n} & \cdots & \beta_{n} & \beta_{n}\\
\frac{1}{\sqrt{n}} & \beta_{n} & \alpha_{n} & \beta_{n} & \beta_{n} & \beta_{n} & \cdots & \beta_{n} & \beta_{n}\\
\frac{1}{\sqrt{n}} & \beta_{n} & \beta_{n} & \alpha_{n} & \beta_{n} & \beta_{n} & \cdots & \beta_{n} & \beta_{n}\\
\frac{1}{\sqrt{n}} & \beta_{n} & \beta_{n} & \beta_{n} & \alpha_{n} & \beta_{n} & \cdots & \beta_{n} & \beta_{n}\\
\frac{1}{\sqrt{n}} & \beta_{n} & \beta_{n} & \beta_{n} & \beta_{n} & \alpha_{n} & \cdots & \beta_{n} & \beta_{n}\\
\vdots & \vdots & \vdots & \vdots & \vdots & \vdots & \ddots & \vdots & \vdots\\
\frac{1}{\sqrt{n}} & \beta_{n} & \beta_{n} & \beta_{n} & \beta_{n} & \beta_{n} & \cdots & \alpha_{n} & \beta_{n}\\
\frac{1}{\sqrt{n}} & \beta_{n} & \beta_{n} & \beta_{n} & \beta_{n} & \beta_{n} & \cdots & \beta_{n} & \alpha_{n}
\end{bmatrix}.
\]
To show that $U_{n}$ is a unitary, the following calculations can
be easily checked: 
\begin{enumerate}
\item $\frac{1}{n}+\alpha_{n}^{2}+(n-2)\cdot\beta_{n}^{2}=1$. 
\item $\frac{1}{\sqrt{n}}+\alpha_{n}+(n-2)\cdot\beta_{n}=0$. 
\item $\frac{1}{n}+2\cdot\alpha_{n}\cdot\beta_{n}+(n-3)\cdot\beta_{n}^{2}=0$. 
\end{enumerate}
Also, to see that $A_{n}=U_{n}T_{n}U_{n}^{*}$, the following calculations
are again easy to check: 
\begin{enumerate}
\item 
\[
A_{n}(1,1)=a_{n}=\langle u_{1},Tu_{1}\rangle=\frac{1}{n}+\frac{1}{n}\cdot(n-2)\cdot r_{n}.
\]
\item 
\[
A_{n}(1,2)=b_{n}=\langle u_{1},Tu_{2}\rangle=\frac{1}{n}+\frac{1}{\sqrt{n}}\cdot(n-2)\cdot r_{n}\cdot\beta_{n}.
\]
\item 
\[
A_{n}(n,1)=b_{n}=\langle u_{n},Tu_{1}\rangle=\frac{1}{n}+\frac{1}{\sqrt{n}}\cdot(n-2)\cdot r_{n}\cdot\beta_{n}.
\]
\item For $3\leq j\leq n$, 
\[
A_{n}(1,j)=0=\langle u_{1},Tu_{j}\rangle=\frac{1}{n}+\frac{1}{\sqrt{n}}\cdot\alpha_{n}\cdot r_{n}+(n-3)\cdot\frac{1}{\sqrt{n}}\cdot\beta_{n}\cdot r_{n}.
\]
\item For $2\leq i\leq n-1$, 
\[
A_{n}(i,1)=0=\langle u_{i},Tu_{1}\rangle=\frac{1}{n}+\frac{1}{\sqrt{n}}\cdot\alpha_{n}\cdot r_{n}+\frac{1}{\sqrt{n}}\cdot(n-3)\cdot\beta_{n}\cdot r_{n}.
\]
\item For $2\leq i\leq n-1$, 
\[
A_{n}(i,2)=c_{n}=\langle u_{i},Tu_{2}\rangle=\frac{1}{n}+\alpha_{n}\cdot\beta_{n}\cdot r_{n}+(n-3)\cdot\beta_{n}^{2}\cdot r_{n}.
\]
\item For $3\leq j\leq n$, 
\[
A_{n}(n,j)=c_{n}=\langle u_{n},Tu_{j}\rangle=\frac{1}{n}+\alpha_{n}\cdot\beta_{n}\cdot r_{n}+(n-3)\cdot\beta_{n}^{2}\cdot r_{n}.
\]
\item For $2\leq i\leq n-1$, 
\[
A_{n}(i,i+1)=d_{n}=\langle u_{i},Tu_{i+1}\rangle=\frac{1}{n}+\alpha_{n}^{2}\cdot r_{n}+(n-3)\cdot\beta_{n}^{2}\cdot r_{n}.
\]
\item For $2\leq i\leq n-2$, $3\leq j\leq n$, $i+1\not=j$, 
\[
A_{n}(i,j)=e_{n}=\langle u_{i},Tu_{j}\rangle=\frac{1}{n}+2\cdot\alpha_{n}\cdot\beta_{n}\cdot r_{n}+(n-4)\cdot\beta_{n}^{2}\cdot r_{n}.
\]
\item 
\[
A_{n}(n,2)=f_{n}=\langle u_{n},Tu_{2}\rangle=\frac{1}{n}+(n-2)\cdot r_{n}\cdot\beta_{n}^{2}.
\]
\end{enumerate}
We thus get a Schur decomposition for $A_{n}$, and since the diagonal
of $T_{n}$ contains only zeros except the trivial eigenvalue 1, we
get that all nontrivial eigenvalues of $A_{n}$ are zero. This completes
the proof of (2).

If we let the set $S=\{1\}$, then we get that 
\[
\phi(A_{n})\leq\phi_{S}(A_{n})=b_{n}<\frac{1}{\sqrt{n}}.
\]
Further, since $T_{n}$ can be written as $\Pi_{n}D_{n}$, where $D_{n}(1,1)=1$,
$D_{n}(i,i)=r_{n}$ for $i=2$ to $n-1$, and $D_{n}(n,n)=0$ for
some permutation $\Pi_{n}$, we get that $A_{n}=(U_{n}\Pi_{n})D_{n}U_{n}^{*}$
which gives a singular value decomposition for $A_{n}$ since $U_{n}\Pi_{n}$
and $U_{n}^{*}$ are unitaries. Thus, $A_{n}$ has exactly one singular
value that is 1, $n-2$ singular values that are $r_{n}$, and one
singular value that is 0. Thus, from Lemma \ref{lem:sigma-bound-phi},
we get that 
\[
\phi(A)\geq\frac{1-r_{n}}{2}=\frac{1}{2\cdot\left(\sqrt{n}+2\right)}\geq\frac{1}{6\sqrt{n}}
\]
and this completes the proof of (3). 
\end{proof}

\section{\label{sec:Proof-of-gen-fied}Proof of the upper bound on $\phi$
in Theorem \ref{thm:gen-phi-lambda2}}
\begin{proof}
Let $R$ be the nonnegative matrix as described, and let $u$ and
$v$ be the eigenvectors corresponding to the eigenvalue $\lambda_{1}=1$.
Let $\kappa=\min_{i}u_{i}\cdot v_{i}$. If $\kappa=0$, $\phi(R)=0$
by Definition \ref{def:gen-edge-expansion} and the upper bound on
$\phi(R)$ in the lemma trivially holds. If $\kappa>0$, then both
$u$ and $v$ are positive, and we normalize them so that $\langle u,v\rangle=1$,
and define 
\[
A=D_{u}^{\frac{1}{2}}D_{v}^{-\frac{1}{2}}RD_{u}^{-\frac{1}{2}}D_{v}^{\frac{1}{2}}
\]
and 
\[
w=D_{u}^{\frac{1}{2}}D_{v}^{\frac{1}{2}}\boldsymbol{1},
\]
where we use positive square roots for the entries in the diagonal
matrices. Further, from Lemma \ref{lem:transRtoA}, the edge expansion
and eigenvalues of $R$ and $A$ are the same, and $\|A\|_{2}=1$.
We will show the bound for $A$ and the bound for $R$ will follow.
Since $A$ has $w$ as both the left and right eigenvector for eigenvalue
1, so does $M=\frac{A+A^{T}}{2}$.

As explained before Definition \ref{def:gen-edge-expansion} for edge
expansion of general nonnegative matrices, $D_{w}AD_{w}$ is Eulerian,
since $D_{w}AD_{w}\mathbf{1}=D_{w}Aw=D_{w}w=D_{w}^{2}\mathbf{1}=D_{w}w=D_{w}A^{T}w=D_{w}A^{T}D_{w}\mathbf{1}$.
Thus, for any $S$, 
\[
\langle\mathbf{1}_{S},D_{w}AD_{w}\mathbf{1}_{\overline{S}}\rangle=\langle\mathbf{1}_{\overline{S}},D_{w}AD_{w}\mathbf{1}_{S}\rangle=\langle\mathbf{1}_{S},D_{w}A^{T}D_{w}\mathbf{1}_{\overline{S}}\rangle,
\]
and thus for any set $S$ for which $\sum_{i\in S}w_{i}^{2}\leq\frac{1}{2}$,
\[
\phi_{S}(A)=\frac{\langle\mathbf{1}_{S},D_{w}AD_{w}\mathbf{1}_{\overline{S}}\rangle}{\langle\mathbf{1}_{S},D_{w}AD_{w}\mathbf{1}\rangle}=\frac{1}{2}\cdot\frac{\langle\mathbf{1}_{S},D_{w}AD_{w}\mathbf{1}_{\overline{S}}\rangle+\langle\mathbf{1}_{S},D_{w}A^{T}D_{w}\mathbf{1}_{\overline{S}}\rangle}{\langle\mathbf{1}_{S},D_{w}AD_{w}\mathbf{1}\rangle}=\phi_{S}(M)
\]
and thus 
\begin{equation}
\phi(A)=\phi(M).\label{eq:up2-1}
\end{equation}
For any matrix $H$, let 
\[
R_{H}(x)=\frac{\langle x,Hx\rangle}{\langle x,x\rangle}.
\]
For every $x\in\mathbb{C}^{n}$, 
\begin{equation}
\text{Re}R_{A}(x)=R_{M}(x),\label{eq:up1-1}
\end{equation}
since $A$ and $\langle x,x\rangle$ are nonnegative and we can write
\begin{align*}
R_{M}(x) & =\frac{1}{2}\frac{\langle x,Ax\rangle}{\langle x,x\rangle}+\frac{1}{2}\frac{\langle x,A^{*}x\rangle}{\langle x,x\rangle}=\frac{1}{2}\frac{\langle x,Ax\rangle}{\langle x,x\rangle}+\frac{1}{2}\frac{\langle Ax,x\rangle}{\langle x,x\rangle}=\frac{1}{2}\frac{\langle x,Ax\rangle}{\langle x,x\rangle}+\frac{1}{2}\frac{\langle x,Ax\rangle^{*}}{\langle x,x\rangle}=\text{Re}R_{A}(x).
\end{align*}
Also, 
\begin{equation}
\text{Re}\lambda_{2}(A)\leq\lambda_{2}(M).\label{eq:up3-1}
\end{equation}
To see this, first note that if $\lambda_{2}(A)=1$, since there exists
some positive $u$ and $v$ for PF eigenvalue 1, then from Lemma \ref{lem:phiR-ind-uv},
the underlying graph of $A$ has multiply strongly connected components
each with eigenvalue 1, and so does $M$, and thus $\lambda_{2}(M)=1$.
If $\lambda_{2}(A)\not=1$, then let $v$ be the eigenvector corresponding
to $\lambda_{2}(A)$. Then since $Aw=w$, we have that 
\[
Av=\lambda_{2}(A)v\Rightarrow\langle w,Av\rangle=\langle w,\lambda_{2}(A)v\rangle\Leftrightarrow\langle A^{T}w,v\rangle=\lambda_{2}(A)\langle w,v\rangle\Leftrightarrow(1-\lambda_{2}(A))\langle w,v\rangle=0
\]
which implies that $v\perp w$. Thus, we have that 
\[
\text{Re}\lambda_{2}(A)=\text{Re}\frac{\langle v,Av\rangle}{\langle v,v\rangle}=\frac{\langle v,Mv\rangle}{\langle v,v\rangle}\leq\max_{u\perp w}\frac{\langle u,Mu\rangle}{\langle u,u\rangle}=\lambda_{2}(M)
\]
where the second equality uses equation \ref{eq:up1-1}, and the last
equality follows from the variational characterization of eigenvalues
stated in the Preliminaries (Appendix \ref{sec:Preliminaries}). Thus,
using equation \ref{eq:up3-1}, equation \ref{eq:up2-1} and Cheeger's
inequality for $M$ (Theorem \ref{thm:gen-Cheeger's-inequalities}),
we get 
\[
\phi(A)=\phi(M)\leq\sqrt{2\cdot(1-\lambda_{2}(M))}\leq\sqrt{2\cdot(1-\text{Re}\lambda_{2}(A))}
\]
as required. 
\end{proof}

\section{Proofs for the lower bound on $\phi$ in Theorem \ref{thm:gen-phi-lambda2}}

\subsection{\label{sec:Proof-of-pow-bound}Proof of Lemma \ref{lem:pow_bound}}
\begin{proof}
For any cut $S$ and non negative $R$ as defined, let $G_{R}=D_{u}RD_{v}$,
and note that $G_{R}$ is Eulerian, since $G_{R}\mathbf{1}=G_{R}^{T}\mathbf{1}$.
Let 
\begin{align}
\gamma_{S}(R) & =\langle\boldsymbol{1}_{\overline{S}},G_{R}\boldsymbol{1}_{S}\rangle+\langle\boldsymbol{1}_{S},G_{R}\boldsymbol{1}_{\overline{S}}\rangle\nonumber \\
 & =\langle\boldsymbol{1}_{\overline{S}},D_{u}RD_{v}\boldsymbol{1}_{S}\rangle+\langle\boldsymbol{1}_{S},D_{u}RD_{v}\boldsymbol{1}_{\overline{S}}\rangle\label{eq:gen-pow-eq}
\end{align}
for any matrix $R$ with left and right eigenvectors $u$ and $v$.
We will show that if $R$ and $B$ are non negative matrices that
have the same left and right eigenvectors $u$ and $v$ for eigenvalue
1, then for every cut $S$, 
\[
\gamma_{S}(RB)\leq\gamma_{S}(R+B)=\gamma_{S}(R)+\gamma_{S}(B).
\]

Fix any cut $S$. Assume $R=\left[\begin{array}{cc}
P & Q\\
H & V
\end{array}\right]$ and $B=\left[\begin{array}{cc}
X & Y\\
Z & W
\end{array}\right]$ naturally divided based on cut $S$. For any vector $u$, let $D_{u}$
be the diagonal matrix with $u$ on the diagonal. Since $Rv=v,R^{T}u=u,Bv=v,B^{T}u=u$,
we have

\begin{equation}
P^{T}D_{u_{S}}\boldsymbol{1}+H^{T}D_{u_{\overline{S}}}\boldsymbol{1}=D_{u_{S}}\boldsymbol{1},\label{eq:gen-pow-1}
\end{equation}
\begin{equation}
ZD_{v_{S}}\boldsymbol{1}+WD_{v_{\overline{S}}}\boldsymbol{1}=D_{v_{\overline{S}}}\boldsymbol{1},\label{eq:gen-pow-2}
\end{equation}
\begin{equation}
XD_{v_{S}}\boldsymbol{1}+YD_{v_{\overline{S}}}\boldsymbol{1}=D_{v_{S}}\boldsymbol{1},\label{eq:gen-pow-3}
\end{equation}
\begin{equation}
Q^{T}D_{u_{S}}\boldsymbol{1}+V^{T}D_{u_{\overline{S}}}\boldsymbol{1}=D_{u_{\overline{S}}}\boldsymbol{1},\label{eq:gen-pow-4}
\end{equation}
where $u$ is divided into $u_{S}$ and $u_{\overline{S}}$ and $v$
into $v_{S}$ and $v_{\overline{S}}$ naturally based on the cut $S$.
Further, in the equations above and in what follows, the vector $\mathbf{1}$
is the all 1's vector with dimension either $|S|$ or $|\overline{S}|$
which should be clear from the context of the equations, and we avoid
using different vectors to keep the notation simpler. Then we have
from equation \ref{eq:gen-pow-eq},

\begin{align*}
\gamma_{S}(R) & =\langle\boldsymbol{1}_{\overline{S}},D_{u}RD_{v}\boldsymbol{1}_{S}\rangle+\langle\boldsymbol{1}_{S},D_{u}RD_{v}\boldsymbol{1}_{\overline{S}}\rangle\\
 & =\langle\boldsymbol{1},D_{u_{S}}QD_{v_{\overline{S}}}\boldsymbol{1}\rangle+\langle\boldsymbol{1},D_{u_{\overline{S}}}HD_{v_{S}}\boldsymbol{1}\rangle
\end{align*}
and similarly 
\[
\gamma_{S}(B)=\langle\boldsymbol{1},D_{u_{S}}YD_{v_{\overline{S}}}\boldsymbol{1}\rangle+\langle\boldsymbol{1},D_{u_{\overline{S}}}ZD_{v_{S}}\boldsymbol{1}\rangle.
\]
The matrix $RB$ also has $u$ and $v$ as the left and right eigenvectors
for eigenvalue 1 respectively, and thus,

\begin{align}
\gamma_{S}(RB) & =\langle\boldsymbol{1},D_{u_{S}}PYD_{v_{\overline{S}}}\boldsymbol{1}\rangle+\langle\boldsymbol{1},D_{u_{S}}QWD_{v_{\overline{S}}}\boldsymbol{1}\rangle+\langle\boldsymbol{1},D_{u_{\overline{S}}}HXD_{v_{S}}\boldsymbol{1}\rangle+\langle\boldsymbol{1},D_{u_{\overline{S}}}VZD_{v_{S}}\boldsymbol{1}\rangle\nonumber \\
 & =\langle P^{T}D_{u_{S}}\boldsymbol{1},YD_{v_{\overline{S}}}\boldsymbol{1}\rangle+\langle\boldsymbol{1},D_{u_{S}}QWD_{v_{\overline{S}}}\boldsymbol{1}\rangle+\langle\boldsymbol{1},D_{u_{\overline{S}}}HXD_{v_{S}}\boldsymbol{1}\rangle+\langle V^{T}D_{u_{\overline{S}}}\boldsymbol{1},ZD_{v_{S}}\boldsymbol{1}\rangle\nonumber \\
 & =\langle D_{u_{S}}\boldsymbol{1}-H^{T}D_{u_{\overline{S}}}\boldsymbol{1},YD_{v_{\overline{S}}}\boldsymbol{1}\rangle+\langle\boldsymbol{1},D_{u_{S}}Q(D_{v_{\overline{S}}}\boldsymbol{1}-ZD_{v_{S}}\boldsymbol{1})\rangle\nonumber \\
 & \ \ \ \ \ +\langle\boldsymbol{1},D_{u_{\overline{S}}}H(D_{v_{S}}\boldsymbol{1}-YD_{v_{\overline{S}}}\boldsymbol{1}\rangle+\langle D_{u_{\overline{S}}}\boldsymbol{1}-Q^{T}D_{u_{S}}\boldsymbol{1},ZD_{v_{S}}\boldsymbol{1}\rangle\nonumber \\
 & \,\,\,\,\,\,\,\,\,\,\,\,\,\,\,\,\,\,\,\,\,\,\text{[from equations \ref{eq:gen-pow-1}, \ref{eq:gen-pow-2}, \ref{eq:gen-pow-3}, \ref{eq:gen-pow-4} above]}\nonumber \\
 & =\langle D_{u_{S}}\boldsymbol{1},YD_{v_{\overline{S}}}\boldsymbol{1}\rangle+\langle\boldsymbol{1},D_{u_{S}}QD_{v_{\overline{S}}}\boldsymbol{1}\rangle-\langle H^{T}D_{u_{\overline{S}}}\boldsymbol{1},YD_{v_{\overline{S}}}\boldsymbol{1}\rangle-\langle\boldsymbol{1},D_{u_{S}}QZD_{v_{S}}\boldsymbol{1}\rangle\nonumber \\
 & \ \ \ \ \ +\langle\boldsymbol{1},D_{u_{\overline{S}}}HD_{v_{S}}\boldsymbol{1}\rangle+\langle D_{u_{\overline{S}}}\boldsymbol{1},ZD_{v_{S}}\boldsymbol{1}\rangle-\langle\boldsymbol{1},D_{u_{\overline{S}}}HYD_{v_{\overline{S}}}\boldsymbol{1}\rangle-\langle Q^{T}D_{u_{S}}\boldsymbol{1},ZD_{v_{S}}\boldsymbol{1}\rangle\nonumber \\
 & \leq\langle D_{u_{S}}\boldsymbol{1},YD_{v_{\overline{S}}}\boldsymbol{1}\rangle+\langle\boldsymbol{1},D_{u_{S}}QD_{v_{\overline{S}}}\boldsymbol{1}\rangle+\langle\boldsymbol{1},D_{u_{\overline{S}}}HD_{v_{S}}\boldsymbol{1}\rangle+\langle D_{u_{\overline{S}}}\boldsymbol{1},ZD_{v_{S}}\boldsymbol{1}\rangle\nonumber \\
 & \,\,\,\,\,\,\,\,\,\,\,\,\,\,\,\,\,\,\,\,\,\,\text{[since every entry of the matrices is nonnegative, and thus each of the terms above]}\nonumber \\
 & =\gamma_{S}(R)+\gamma_{S}(B)\label{eq:gamma_bound-1}
\end{align}
Thus, noting that $R^{k}$ has $u$ and $v$ as left and right eigenvectors
for any $k$, we inductively get using inequality \ref{eq:gamma_bound-1}
that 
\[
\gamma_{S}(R^{k})\leq\gamma_{S}(R)+\gamma_{S}(R^{k-1})\leq\gamma_{S}(R)+(k-1)\cdot\gamma_{S}(R)=k\cdot\gamma_{S}(R).
\]
Note that for any fixed set $S$, the denominator in the definition
of $\phi_{S}(R)$ is independent of the matrix $R$ and depends only
on the set $S$, and since the numerator is exactly $\gamma_{S}(R)/2$,
i.e. 
\[
\phi_{S}(R)=\frac{\gamma_{S}(R)}{2\cdot\sum_{i\in S}u_{i}\cdot v_{i}}
\]
we get by letting $S$ be the set that minimizes $\phi(R)$, that
\[
\phi(R^{k})\leq k\cdot\phi(R).
\]
\end{proof}

\subsection{\label{sec:Proof-of-T-norm-bound}Proof of Lemma \ref{lem:T_norm_bound}}
\begin{proof}
Let $g_{r}(k)$ denote the maximum of the absolute value of entries
at distance $r$ from the diagonal in $T^{k}$, where the diagonal
is at distance 0 from the diagonal, the off-diagonal is at distance
1 from the diagonal and so on. More formally, 
\[
g_{r}(k)=\max_{i}|T^{k}(i,i+r)|.
\]
We will inductively show that for $\alpha\leq1$, and $r\geq1$, 
\begin{equation}
g_{r}(k)\leq{k+r \choose r}\cdot\alpha^{k-r}\cdot\sigma^{r},\label{eq:lo2-1-1}
\end{equation}
where $\sigma=\|T\|_{2}$. First note that for $r=0$, since $T$
is upper triangular, the diagonal of $T^{k}$ is $\alpha^{k}$, and
thus the hypothesis holds for $r=0$ and all $k\geq1$. Further, for
$k=1$, if $r=0$, then $g_{0}(1)\leq\alpha$ and if $r\geq1$, then
$g_{r}(1)\leq\|T\|_{2}\leq\sigma$ and the inductive hypothesis holds
also in this case, since $r\geq k$ and $\alpha^{k-r}\geq1$. For
the inductive step, assume that for all $r\geq1$ and all $j\leq k-1$,
$g_{r}(j)\leq{j+r \choose r}\cdot\alpha^{j-r}\cdot\sigma^{r}$. We
will show the calculation for $g_{r}(k)$.

Since $|a+b|\leq|a|+|b|$, 
\begin{align*}
g_{r}(k) & \leq\sum_{i=0}^{r}g_{r-i}(1)\cdot g_{i}(k-1)\\
 & =g_{0}(1)\cdot g_{r}(k-1)+\sum_{i=0}^{r-1}g_{r-i}(1)\cdot g_{i}(k-1).
\end{align*}
The first term can be written as, 
\begin{align}
g_{0}(1)\cdot g_{r}(k-1) & =\alpha\cdot{k-1+r \choose r}\cdot\alpha^{k-1-r}\cdot\sigma^{r}\nonumber \\
 & \ \ \ \ \ \ \ \ \text{[using that \ensuremath{g_{0}(1)\leq\alpha} and the inductive hypothesis for the second term]}\nonumber \\
 & \leq\alpha^{k-r}\cdot\sigma^{r}\cdot{k+r \choose r}\cdot\dfrac{{k-1+r \choose r}}{{k+r \choose r}}\nonumber \\
 & \leq\frac{k}{k+r}\cdot{k+r \choose r}\cdot\alpha^{k-r}\cdot\sigma^{r}\label{eq:inproofeq-1}
\end{align}
and the second term as 
\begin{align*}
g_{r}(k) & \leq\sum_{i=0}^{r-1}g_{r-i}(1)\cdot g_{i}(k-1)\\
 & \leq\sigma\cdot\sum_{i=0}^{r-1}{k-1+i \choose i}\cdot\alpha^{k-1-i}\cdot\sigma^{i}\\
 & \ \ \ \ \ \text{[using \ensuremath{g_{r-i}(1)\leq\sigma} and the inductive hypothesis for the second term]}\\
 & \le\sigma^{r}\cdot\alpha^{k-r}\cdot{k+r \choose r}\cdot\sum_{i=0}^{r-1}\cdot\dfrac{{k-1+i \choose i}}{{k+r \choose r}}\cdot\alpha^{r-1-i}\\
 & \ \ \ \ \ \text{[using \ensuremath{\sigma^{i}\leq\sigma^{r-1}} since \ensuremath{\sigma\geq1} and \ensuremath{i\leq r-1}]}\\
 & \leq\sigma^{r}\cdot\alpha^{k-r}\cdot{k+r \choose r}\cdot\sum_{i=0}^{r-1}\cdot\dfrac{{k-1+i \choose i}}{{k+r \choose r}}\\
 & \ \ \ \ \ \text{[using \ensuremath{\alpha\leq1} and \ensuremath{i\leq r-1}]}
\end{align*}
We will now show that the quantity inside the summation is at most
$\frac{r}{k+r}$. Inductively, for $r=1$, the statement is true,
and assume that for any other $r$, 
\[
\sum_{i=0}^{r-1}\frac{{k-1+i \choose i}}{{k+r \choose r}}\leq\frac{r}{k+r}.
\]
Then we have 
\begin{align*}
\sum_{i=0}^{r}\frac{{k-1+i \choose i}}{{k+r+1 \choose r+1}} & =\frac{{k+r \choose r}}{{k+r+1 \choose r+1}}\cdot\sum_{i=0}^{r-1}\frac{{k-1+i \choose i}}{{k+r \choose r}}+\frac{{k-1+r \choose k-1}}{{k+r+1 \choose k}}\\
 & \leq\frac{r+1}{k+r+1}\cdot\frac{r}{k+r}+\frac{(r+1)\cdot k}{(k+r+1)\cdot(k+r)}\\
 & =\frac{r+1}{k+r+1}
\end{align*}
Thus we get that the second term is at most $\sigma^{r}\cdot\alpha^{k-r}\cdot{k+r \choose r}\cdot\frac{r}{k+r}$,
and combining it with the first term (equation \ref{eq:inproofeq-1}),
it completes the inductive hypothesis. Noting that the operator norm
is at most the Frobenius norm, and since $g_{r}(k)$ is increasing
in $r$ and the maximum value of $r$ is $n$, we get using equation
\ref{eq:lo2-1-1}, 
\begin{align*}
\|T^{k}\|_{2} & \leq\sqrt{\sum_{i,j}|T^{k}(i,j)|^{2}}\\
 & \leq n\cdot\sigma^{n}\cdot\alpha^{k-n}\cdot{k+n \choose n}
\end{align*}
as required. 
\end{proof}

\subsection{\label{sec:Proof-of-T-norm-1-bound}Proof of Lemma \ref{lem:T_norm_1_bound}}
\begin{proof}
Let $X=T^{k_{1}}$, where $k_{1}=\frac{c_{1}}{1-\alpha}$ for $\alpha<1$.
Then $\|X\|_{2}\leq\|T\|_{2}^{k_{1}}\leq1$, and for every $i$, 
\[
|X_{i,i}|\leq|T_{i,i}|^{k_{1}}\leq|\lambda_{m}|^{c_{1}/(1-\alpha)}\leq e^{-c_{1}}.
\]
Using Lemma \ref{lem:T_norm_bound} for $X$ with $\sigma=1$ and
$\beta=e^{-c_{1}}$, we get that for $k_{2}=c_{2}\cdot n$, 
\begin{align*}
\|X^{k_{2}}\| & \leq n\cdot{k_{2}+n \choose n}\cdot e^{-c_{1}(k_{2}-n)}\\
 & \leq n\cdot e^{n}\cdot(c_{2}+1)^{n}\cdot e^{-c_{1}(c_{2}-1)n}\\
 & \ \ \ \ \ \ \ \text{[using \ensuremath{{a \choose b}\leq\left(\frac{ea}{b}\right)^{b}}]}\\
 & =\exp\left(n\cdot\left(\frac{\ln n}{n}+\ln(c_{2}+1)+1+c_{1}-c_{1}c_{2}\right)\right)
\end{align*}
and to have this quantity less than $\epsilon$, we require 
\begin{align}
 & \exp\left(n\cdot\left(\frac{\ln n}{n}+\ln(c_{2}+1)+1+c_{1}-c_{1}c_{2}\right)\right)\leq\epsilon\nonumber \\
\Leftrightarrow & \left(\frac{1}{\epsilon}\right)^{\frac{1}{n}}\leq\exp\left(-1\cdot\left(\frac{\ln n}{n}+\ln(c_{2}+1)+1+c_{1}-c_{1}c_{2}\right)\right)\nonumber \\
\Leftrightarrow & \frac{1}{n}\ln\frac{n}{\epsilon}+1+c_{1}+\ln(c_{2}+1)\leq c_{1}c_{2}\label{eq:proofmodbound-1}
\end{align}
and we set 
\[
c_{1}=1+\frac{1}{2.51}\cdot\frac{1}{n}\ln\left(\frac{n}{\epsilon}\right)
\]
and $c_{2}=3.51$ which always satisfies inequality \ref{eq:proofmodbound-1}.
As a consequence, for 
\begin{align*}
k & =k_{1}\cdot k_{2}\\
 & =\frac{c_{1}\cdot c_{2}\cdot n}{1-\alpha}\\
 & =\dfrac{3.51\cdot n+1.385\cdot\ln\left(\frac{n}{\epsilon}\right)}{1-\alpha}
\end{align*}
we get, 
\[
\|T^{k}\|_{2}\leq\|X^{k_{2}}\|\leq\epsilon
\]
as required. 
\end{proof}

\subsection{\label{sec:Proof-of-transRtoA}Proof of Lemma \ref{lem:transRtoA}}
\begin{proof}
Let the matrix $A$ be as defined, and let $w=D_{u}^{\frac{1}{2}}D_{v}^{\frac{1}{2}}\mathbf{1}$.
Then it is easily checked that $Aw=w$ and $A^{T}w=w$. Further, 
\[
\langle w,w\rangle=\langle D_{u}^{\frac{1}{2}}D_{v}^{\frac{1}{2}}\mathbf{1},D_{u}^{\frac{1}{2}}D_{v}^{\frac{1}{2}}\mathbf{1}\rangle=\langle D_{u}\mathbf{1},D_{v}\mathbf{1}\rangle=\langle u,v\rangle=1
\]
where we used the fact that the matrices $D_{u}^{\frac{1}{2}}$ and
$D_{v}^{\frac{1}{2}}$ are diagonal, and so they commute, and are
unchanged by taking transposes. Let $S$ be any set. The condition
$\sum_{i\in S}u_{i}\cdot v_{i}\leq\frac{1}{2}$ translates to $\sum_{i\in S}w_{i}^{2}\leq\frac{1}{2}$
since $u_{i}\cdot v_{i}=w_{i}^{2}$. Thus, for any set $S$ for which
$\sum_{i\in S}u_{i}\cdot v_{i}=\sum_{i\in S}w_{i}^{2}\leq\frac{1}{2}$,
\begin{align*}
\phi_{S}(R) & =\frac{\langle\mathbf{1}_{S},D_{u}RD_{v}\mathbf{1}_{\overline{S}}\rangle}{\langle\mathbf{1}_{S},D_{u}D_{v}\mathbf{1}_{S}\rangle}\\
 & =\frac{\langle\mathbf{1}_{S},D_{u}D_{u}^{-\frac{1}{2}}D_{v}^{\frac{1}{2}}AD_{u}^{\frac{1}{2}}D_{v}^{-\frac{1}{2}}D_{v}\mathbf{1}_{\overline{S}}\rangle}{\langle\mathbf{1}_{S},D_{u}^{\frac{1}{2}}D_{v}^{\frac{1}{2}}D_{u}^{\frac{1}{2}}D_{v}^{\frac{1}{2}}\mathbf{1}_{S}\rangle}\\
 & =\frac{\langle\mathbf{1}_{S},D_{u}^{\frac{1}{2}}D_{v}^{\frac{1}{2}}AD_{u}^{\frac{1}{2}}D_{v}^{\frac{1}{2}}\mathbf{1}_{\overline{S}}\rangle}{\langle\mathbf{1}_{S},D_{u}^{\frac{1}{2}}D_{v}^{\frac{1}{2}}D_{u}^{\frac{1}{2}}D_{v}^{\frac{1}{2}}\mathbf{1}_{S}\rangle}\\
 & =\frac{\langle\mathbf{1}_{S},D_{w}AD_{w}\mathbf{1}_{\overline{S}}\rangle}{\langle\mathbf{1}_{S},D_{w}^{2}\mathbf{1}_{S}\rangle}\\
 & =\phi_{S}(A)
\end{align*}
and (1) holds. Further, since $A$ is a similarity transform of $R$,
all eigenvalues are preserved and (2) holds. For (3), consider the
matrix $H=A^{T}A$. Since $w$ is the positive left and right eigenvector
for $A$, i.e. $Aw=w$ and $A^{T}w=w$, we have $Hw=w$. But since
$A$ was nonnegative, so is $H$, and since it has a positive eigenvector
$w$ for eigenvalue 1, by Perron-Frobenius (Theorem \ref{thm:(Perron-Frobenius)},
part 2), $H$ has PF eigenvalue 1. But $\lambda_{i}(H)=\sigma_{i}^{2}(A)$,
where $\sigma_{i}(A)$ is the $i$'th largest singular value of $A$.
Thus, we get $\sigma_{1}^{2}(A)=\lambda_{1}(H)=1$, and thus $\|A\|_{2}=1$. 
\end{proof}

\subsection{\label{sec:Proof-of-mod-lambda-bound}Proof of Lemma \ref{lem:mod_lambda_bound}}
\begin{proof}
From the given $R$ with positive left and right eigenvectors $u$
and $v$ for eigenvalue 1 as stated, let $A=D_{u}^{\frac{1}{2}}D_{v}^{-\frac{1}{2}}RD_{u}^{-\frac{1}{2}}D_{v}^{\frac{1}{2}}$
and $w=D_{u}^{\frac{1}{2}}D_{v}^{\frac{1}{2}}\mathbf{1}$ as in Lemma
\ref{lem:transRtoA}. Note that $w$ is positive, and 
\[
\langle w,w\rangle=\langle D_{u}^{\frac{1}{2}}D_{v}^{\frac{1}{2}}\mathbf{1},D_{u}^{\frac{1}{2}}D_{v}^{\frac{1}{2}}\mathbf{1}\rangle=\langle u,v\rangle=1.
\]
Further, $Aw=w$ and $A^{T}w=w$.

Let $A=w\cdot w^{T}+B$. Since $(w\cdot w^{T})^{2}=w\cdot w^{T}$
and $Bw\cdot w^{T}=w\cdot w^{T}B=0$, we get 
\[
A^{k}=w\cdot w^{T}+B^{k}.
\]
Let $B=UTU^{*}$ be the Schur decomposition of $B$, where the diagonal
of $T$ contains all but the stochastic eigenvalue of $A$, which
is replaced by 0, since $w$ is both the left and right eigenvector
for eigenvalue 1 of $A$, and that space is removed in $w\cdot w^{T}$.
Further, the maximum diagonal entry of $T$ is at most $|\lambda_{m}|$
where $\lambda_{m}$ is the nontrivial eigenvalue of $A$ (or $R$)
that is maximum in magnitude. Note that if $|\lambda_{m}|=1$, then
the lemma is trivially true since $\phi(R)\geq0$, and thus assume
that $|\lambda_{m}|<1.$ Since $w\cdot w^{T}B=Bw\cdot w^{T}=0$ and
$\|A\|_{2}\leq1$ from Lemma \ref{lem:transRtoA}, we have that $\|B\|_{2}\leq1$.

Thus, using Lemma \ref{lem:T_norm_1_bound} (in fact, the last lines
in the proof of Lemma \ref{lem:T_norm_1_bound} in Appendix \ref{sec:Proof-of-T-norm-1-bound}
above), for 
\[
k\geq\dfrac{3.51\cdot n+1.385\cdot\ln\left(\frac{n}{\epsilon}\right)}{1-|\lambda_{m}|},
\]
we get that 
\[
\|B^{k}\|_{2}=\|T^{k}\|_{2}\leq\epsilon,
\]
and for $e_{i}$ being the vector with 1 at position $i$ and zeros
elsewhere, we get using Cauchy-Schwarz 
\[
|B^{k}(i,j)|=\left|\langle e_{i},B^{k}e_{j}\rangle\right|\leq\|e_{i}\|_{2}\|B\|_{2}\|e_{j}\|_{2}\leq\epsilon.
\]
Thus, for any set $S$ for which $\sum_{i\in S}w_{i}^{2}\leq\frac{1}{2}$,
let 
\[
\epsilon=c\cdot\kappa=c\cdot\min_{i}u_{i}\cdot v_{i}=c\cdot\min_{i}w_{i}^{2}\leq c\cdot w_{i}\cdot w_{j}
\]
for any $i,j$, where $c$ is some constant to be set later. Then
\begin{align*}
\phi_{S}(A^{k}) & =\frac{\langle\mathbf{1}_{S},D_{w}A^{k}D_{w}\mathbf{1}_{\overline{S}}\rangle}{\langle\mathbf{1}_{S},D_{w}D_{w}\mathbf{1}\rangle}\\
 & =\frac{\sum_{i\in S,j\in\overline{S}}A^{k}(i,j)\cdot w_{i}\cdot w_{j}}{\sum_{i\in S}w_{i}^{2}}\\
 & =\frac{\sum_{i\in S,j\in\overline{S}}\left(w\cdot w^{T}(i,j)+B^{k}(i,j)\right)\cdot w_{i}\cdot w_{j}}{\sum_{i\in S}w_{i}^{2}}\\
 & \geq\frac{\sum_{i\in S,j\in\overline{S}}\left(w_{i}\cdot w_{j}-\epsilon\right)\cdot w_{i}\cdot w_{j}}{\sum_{i\in S}w_{i}^{2}}\\
 & \geq(1-c)\frac{\sum_{i\in S}w_{i}^{2}\sum_{j\in\overline{S}}w_{j}^{2}}{\sum_{i\in S}w_{i}^{2}}\\
 & \geq\frac{1}{2}(1-c)
\end{align*}
since $\sum_{i\in\overline{S}}w_{i}^{2}\geq\frac{1}{2}.$ Thus, we
get that for 
\[
k\geq\dfrac{3.51n+1.385\cdot\ln\left(\frac{c\cdot n}{\kappa}\right)}{1-|\lambda_{m}|}
\]
the edge expansion 
\[
\phi(A^{k})\geq\frac{1}{2}(1-c),
\]
and thus using Lemma \ref{lem:pow_bound}, we get that 
\begin{align*}
\phi(A) & \geq\frac{1}{k}\cdot\phi(A^{k})\geq\frac{1-c}{2}\cdot\frac{1-|\lambda_{m}|}{3.51\cdot n+1.385\cdot\ln\left(c\cdot n\right)+1.385\cdot\ln\left(\frac{1}{\kappa}\right)}
\end{align*}
and setting $c=\frac{1}{1.4\cdot e}$, and using $\ln(e\cdot x)\leq x,$we
get that 
\begin{align*}
\phi(A) & \geq\frac{1-|\lambda_{m}|}{15\cdot n+2\cdot\ln\left(\frac{1}{\kappa}\right)}\\
 & \geq\frac{1}{15}\cdot\frac{1-|\lambda_{m}|}{n+\ln\left(\frac{1}{\kappa}\right)}
\end{align*}
\end{proof}

\subsection{\label{sec:Proof-of-real-lambda-bound}Proof of Lemma \ref{lem:real-lambda-bound}}
\begin{proof}
Let $R$ be a nonnegative matrix with PF eigenvalue 1, and let some
eigenvalue $\lambda_{r}(R)=\alpha+i\beta$, and letting the gap $g=1-\alpha$,
we can rewrite 
\[
\lambda_{r}=(1-g)+i\beta
\]
with $(1-g)^{2}+\beta^{2}\leq1$, or 
\begin{equation}
g^{2}+\beta^{2}\leq2g.\label{eq:lo3}
\end{equation}
Consider the lazy random walk $\tilde{R}=pI+(1-p)R$ for some $0\leq p\leq1$,
the eigenvalue modifies as 
\begin{align*}
\lambda_{r}(\tilde{R}) & =p+(1-p)(1-g)+i(1-p)\beta\\
 & =1-g(1-p)+i\beta(1-p)
\end{align*}
and for any set $S$, since $\phi_{S}(I)=0,$ we have 
\[
\phi_{S}(\tilde{R})=(1-p)\phi_{S}(R)
\]
or 
\begin{equation}
\phi(\tilde{R})=(1-p)\phi(R).\label{eq:finb1}
\end{equation}
Further, 
\begin{align}
1-|\lambda_{r}(\tilde{R})| & =1-\sqrt{(1-g(1-p))^{2}+(\beta(1-p))^{2}}\nonumber \\
 & =1-\sqrt{1+(g^{2}+\beta^{2})(1-p)^{2}-2g(1-p)}\nonumber \\
 & \geq1-\sqrt{1+2g(1-p)^{2}-2g(1-p)}\nonumber \\
 & \ \ \ \ \ \text{[using inequality \ref{eq:lo3}]}\nonumber \\
 & =1-\sqrt{1-2gp(1-p)}\nonumber \\
 & \geq1-e^{-gp(1-p)}\nonumber \\
 & \ \ \ \ \ \text{[using \ensuremath{1-x\leq e^{-x}}]}\nonumber \\
 & \geq1-(1-\frac{1}{2}gp(1-p))\nonumber \\
 & \ \ \ \ \ \text{[using \ensuremath{e^{-x}\leq1-\frac{1}{2}x} for \ensuremath{0\leq x\leq1}]}\nonumber \\
 & =\frac{1}{2}gp(1-p).\label{eq:lo1}
\end{align}
Thus, we have that, 
\begin{align*}
(1-\text{Re}\lambda_{2}(R))\cdot p(1-p) & \leq(1-\text{Re}\lambda_{m}(R))\cdot p(1-p)\\
 & \ \ \ \ \ \text{[since both \ensuremath{\lambda_{2}} and \ensuremath{\lambda_{m}} are nontrivial eigenvalues]}\\
 & \leq2\cdot(1-|\lambda_{m}(\tilde{R})|)\\
 & \ \ \ \ \ \text{[using inequality \ref{eq:lo1}]}\\
 & \leq2\cdot15\cdot\left(n+\ln\left(\frac{1}{\kappa}\right)\right)\cdot\phi(\tilde{R})\\
 & \ \ \ \ \ \text{[using Lemma \ref{lem:mod_lambda_bound} for \ensuremath{\tilde{R}}]}\\
 & \leq2\cdot15\cdot\left(n+\ln\left(\frac{1}{\kappa}\right)\right)\cdot(1-p)\phi(R)\\
 & \ \ \ \ \ \text{[from equation \ref{eq:finb1}]}
\end{align*}
and taking the limit as $p\rightarrow1$, we get that 
\[
\frac{1}{30}\cdot\frac{1-\text{Re}\lambda_{2}(R)}{n+\ln\left(\frac{1}{\kappa}\right)}\leq\phi(A).
\]
Further, if $A$ is doubly stochastic, then $\kappa=\frac{1}{n}$,
and using $\ln(en)\leq n$ and the last equations at the end of the
proof of Lemma \ref{lem:T_norm_1_bound} in Appendix \ref{sec:Proof-of-T-norm-1-bound}
above, we get using the same calculations that 
\[
\frac{1}{35}\cdot\frac{1-\text{Re}\lambda_{2}(R)}{n}\leq\phi(A)
\]
as required. 
\end{proof}

\section{Proofs of bounds on Mixing Time}

\subsection{\label{sec:Proof-of-mixt-sing-bound}Proof of Lemma \ref{lem:mixt-singval}}
\begin{proof}
Writing $\tau$ as shorthand for $\tau_{\epsilon}(A)$ and since $A=w\cdot w^{T}+B$
with $Bw=0$ and $B^{T}w=0$, we have that $A^{\tau}=w\cdot w^{T}+B^{\tau}$.
Let $x$ be a nonnegative vector such that $\langle w,x\rangle=\|D_{w}x\|_{1}=1$.
As discussed after Definition \ref{def:gen-Mixing-time}, this is
sufficient for bounding mixing time for all $x$. Then we have 
\begin{align*}
\|D_{w}(A^{\tau}-w\cdot w^{T})x\|_{1} & =\|D_{w}B^{\tau}x\|_{1}=\|D_{w}B^{\tau}D_{w}^{-1}y\|_{1}\leq\|D_{w}\|_{1}\|B^{\tau}\|_{1}\|D_{w}^{-1}\|_{1}\|y\|_{1}
\end{align*}
where $y=D_{w}x$ and $\|y\|_{1}=1$. Further, since $\|w\|_{2}=1$
and $\kappa=\min_{i}w_{i}^{2}$, we have $\|D_{w}\|_{1}\leq1$ and
$\|D_{w}^{-1}\|_{1}\leq\frac{1}{\sqrt{\kappa}}$, and using these
bounds to continue the inequalities above, we get 
\[
\|D_{w}(A^{\tau}-w\cdot w^{T})x\|_{1}\leq\frac{1}{\sqrt{\kappa}}\|B^{\tau}\|_{1}\leq\frac{\sqrt{n}}{\sqrt{\kappa}}\|B^{\tau}\|_{2}\leq\frac{\sqrt{n}}{\sqrt{\kappa}}\|B\|_{2}^{\tau}\leq\frac{\sqrt{n}}{\sqrt{\kappa}}\left(\sigma_{2}(A)\right)^{\tau}=\frac{\sqrt{n}}{\sqrt{\kappa}}\left(\sigma_{2}^{c}(A)\right)^{\frac{\tau}{c}}\leq\epsilon
\]
where the second inequality is Cauchy-Schwarz, the fourth inequality
used $\|B\|_{2}\leq\sigma_{2}(A)$ as shown in the proof (in Appendix
\ref{sec:Proof-of-sigmabound}) of Lemma \ref{lem:sigma-bound-phi},
and the last inequality was obtained by setting $\tau=\tau_{\epsilon}(A)=\dfrac{c\cdot\ln\left(\frac{\sqrt{n}}{\sqrt{\kappa}\cdot\epsilon}\right)}{1-\sigma_{2}^{c}(A)}$. 
\end{proof}

\subsection{\label{sec:Proof-of-mixt-phi-rel}Proof of Lemma \ref{lem:mixt-phi-rel}}
\begin{proof}
We first show that the mixing time of $R$ and $A=D_{u}^{\frac{1}{2}}D_{v}^{-\frac{1}{2}}RD_{u}^{-\frac{1}{2}}D_{v}^{\frac{1}{2}}$
are the same. Let $w=D_{u}^{\frac{1}{2}}D_{v}^{\frac{1}{2}}\mathbf{1}$,
the left and right eigenvector of $A$ for eigenvalue 1. We will show
that for every $x$ for which $\|D_{u}x\|_{1}=1$, there exists some
$y$ with $\|D_{w}y\|_{1}=1$ such that 
\[
\|D_{u}(R^{\tau}-v\cdot u^{T})x\|_{1}=\|D_{w}(A^{\tau}-w\cdot w^{T})y\|_{1}
\]
which would imply that $\tau_{\epsilon}(R)=\tau_{\epsilon}(A)$ by
definition.

Let $x$ be a vector with $\|D_{u}x\|_{1}=1$ and let $y=D_{u}^{\frac{1}{2}}D_{v}^{-\frac{1}{2}}x$.
Then since $D_{w}=D_{u}^{\frac{1}{2}}D_{v}^{\frac{1}{2}},$ we get
$\|D_{w}y\|_{1}=\|D_{u}x\|_{1}=1$. Let $R=v\cdot u^{T}+B_{R}$ and
$A=w\cdot w^{T}+B_{A}$ where $B_{A}=D_{u}^{\frac{1}{2}}D_{v}^{-\frac{1}{2}}B_{R}D_{u}^{-\frac{1}{2}}D_{v}^{\frac{1}{2}}$
and $B_{A}^{\tau}=D_{u}^{\frac{1}{2}}D_{v}^{-\frac{1}{2}}B_{R}^{\tau}D_{u}^{-\frac{1}{2}}D_{v}^{\frac{1}{2}}$.
Then 
\begin{align*}
D_{u}(R^{\tau}-v\cdot u^{T})x & =D_{u}B_{R}^{\tau}x\\
 & =D_{u}(D_{u}^{-\frac{1}{2}}D_{v}^{\frac{1}{2}}B_{A}^{\tau}D_{u}^{\frac{1}{2}}D_{v}^{-\frac{1}{2}})x\\
 & =D_{w}B_{A}^{\tau}y\\
 & =D_{w}(A^{\tau}-w\cdot w^{T})y
\end{align*}
as required. Further, from Lemma \ref{lem:transRtoA}, we have that
$\phi(R)=\phi(A)$. Thus, we show the bound for $\tau_{\epsilon}(A)$
and $\phi(A)$, and the bound for $R$ will follow. Note that if $R$
is $\frac{1}{2}$-lazy, then $A$ is also $\frac{1}{2}$-lazy, since
if $R=\frac{1}{2}I+\frac{1}{2}C$ where $C$ is nonnegative, then
\[
A=\frac{1}{2}D_{u}^{\frac{1}{2}}D_{v}^{-\frac{1}{2}}ID_{u}^{-\frac{1}{2}}D_{v}^{\frac{1}{2}}+D_{u}^{\frac{1}{2}}D_{v}^{-\frac{1}{2}}CD_{u}^{-\frac{1}{2}}D_{v}^{\frac{1}{2}}=\frac{1}{2}I+D_{u}^{\frac{1}{2}}D_{v}^{-\frac{1}{2}}CD_{u}^{-\frac{1}{2}}D_{v}^{\frac{1}{2}}.
\]

We first lower bound $\tau_{\epsilon}(A)$ using $\phi$, showing
the lower bound on mixing time in terms of the edge expansion. We
will show the bound for nonnegative vectors $x$, and by Definition
\ref{def:gen-Mixing-time} and the discussion after, it will hold
for all $x$. By definition of mixing time, we have that for any nonnegative
$x$ such that $\|D_{w}x\|_{1}=\langle w,x\rangle=1$, since $A^{\tau}=w\cdot w^{T}+B^{\tau}$,
\[
\|D_{w}(A^{\tau}-w\cdot w^{T})x\|_{1}=\|D_{w}B^{\tau}x\|_{1}\leq\epsilon
\]
and letting $y=D_{w}x$, we get that for any nonnegative $y$ with
$\|y\|_{1}=1$, we have 
\[
\|D_{w}B^{\tau}D_{w}^{-1}y\|_{1}\leq\epsilon.
\]
Plugging the standard basis vectors for $i$, we get that for every
$i$, 
\[
\sum_{j}\left|\frac{1}{w(i)}\cdot B^{\tau}(j,i)\cdot w(j)\right|=\frac{1}{w(i)}\cdot\sum_{j}\left|B^{\tau}(j,i)\right|\cdot w(j)\leq\epsilon.
\]
Thus, for any set $S$, 
\begin{equation}
\sum_{i\in S}w(i)^{2}\cdot\frac{1}{w(i)}\cdot\sum_{j}\left|B^{\tau}(j,i)\right|\cdot w(j)=\sum_{i\in S}\sum_{j}w(i)\cdot\left|B^{\tau}(j,i)\right|\cdot w(j)\leq\sum_{i\in S}w(i)^{2}\cdot\epsilon.\label{eq:proof-mixtphi-1}
\end{equation}
Thus, for any set $S$ for which $\sum_{i\in S}w_{i}^{2}\leq\frac{1}{2}$,
\begin{align*}
\phi_{S}(A^{\tau}) & =\frac{\langle\mathbf{1}_{S},D_{w}A^{\tau}D_{w}\mathbf{1}_{\overline{S}}\rangle}{\langle\mathbf{1}_{S},D_{w}^{2}\mathbf{1}\rangle}=\frac{\langle\mathbf{1}_{\overline{S}},D_{w}A^{\tau}D_{w}\mathbf{1}_{S}\rangle}{\langle\mathbf{1}_{S},D_{w}^{2}\mathbf{1}\rangle}\\
 & \ \ \ \ \text{[since \ensuremath{D_{w}A^{\tau}D_{w}} is Eulerian, i.e. \ensuremath{D_{w}A^{\tau}D_{w}\mathbf{1}=\ensuremath{D_{w}(A^{\tau})^{T}D_{w}}\mathbf{1}}]}\\
 & =\frac{\sum_{i\in S}\sum_{j\in\overline{S}}A^{\tau}(j,i)\cdot w(i)\cdot w(j)}{\sum_{i\in S}w_{i}^{2}}\\
 & =\frac{\sum_{i\in S}\sum_{j\in\overline{S}}w(j)\cdot w(i)\cdot w(i)\cdot w(j)+\sum_{i\in S}\sum_{j\in\overline{S}}B^{\tau}(j,i)\cdot w(i)\cdot w(j)}{\sum_{i\in S}w_{i}^{2}}\\
 & \geq\sum_{j\in\overline{S}}w_{j}^{2}-\frac{\sum_{i\in S}\sum_{j\in\overline{S}}|B^{\tau}(j,i)|\cdot w(i)\cdot w(j)}{\sum_{i\in S}w_{i}^{2}}\\
 & \geq\sum_{j\in\overline{S}}w_{j}^{2}-\frac{\sum_{i\in S}\sum_{j}|B^{\tau}(j,i)|\cdot w(i)\cdot w(j)}{\sum_{i\in S}w_{i}^{2}}\\
 & \geq\frac{1}{2}-\epsilon\\
 & \ \ \ \ \text{[since \ensuremath{\sum_{i\in S}w_{i}^{2}\leq\frac{1}{2}} and \ensuremath{\sum_{i}w_{i}^{2}=1}, and the second term follows from equation \ref{eq:proof-mixtphi-1}]}
\end{align*}
and thus 
\[
\phi(A^{\tau})\geq\frac{1}{2}-\epsilon,
\]
and using Lemma \ref{lem:pow_bound}, we obtain 
\[
\phi(A^{\tau})\leq\tau\cdot\phi(A),
\]
or 
\[
\frac{\frac{1}{2}-\epsilon}{\phi(A)}\leq\tau_{\epsilon}(A).
\]

We now upper bound $\tau_{\epsilon}(A)$ in terms of $\phi(A)$. From
Lemma \ref{lem:mixt-singval} for $c=2$, we have that 
\begin{equation}
\tau_{\epsilon}(A)\leq\frac{2\cdot\ln\left(\frac{n}{\kappa\cdot\epsilon}\right)}{1-\sigma_{2}^{2}(A)}=\frac{2\cdot\ln\left(\frac{n}{\kappa\cdot\epsilon}\right)}{1-\lambda_{2}(AA^{T})}.\label{eq:proof-mixtphi-2}
\end{equation}
The continuing argument is straightforward, and was shown in Fill
\cite{fill1991eigenvalue} (albeit for row-stochastic matrices), and
we reproduce it here for completeness. Since $A$ is $\frac{1}{2}$-lazy,
we have that $2A-I$ is nonnegative, has PF eigenvalue 1, and has
the same left and right eigenvector $w$ for eigenvalue 1 implying
that the PF eigenvalue is 1 by Perron-Frobenius (Theorem \ref{thm:(Perron-Frobenius)},
part 2), and also that its largest singular value is 1 from Lemma
\ref{lem:transRtoA}. Further, $\frac{1}{2}(A+A^{T})$ also has the
same properties. Thus, for any $x$, 
\begin{align*}
AA^{T} & =\frac{A+A^{T}}{2}+\frac{(2A-I)(2A^{T}-I)}{4}-\frac{I}{4}\\
\Rightarrow\langle x,AA^{T}x\rangle & \leq\langle x,\frac{A+A^{T}}{2}x\rangle+\|x\|_{2}\frac{\|(2A-I)\|_{2}\|(2A^{T}-I)\|_{2}}{4}\|x\|_{2}-\frac{\|x\|_{2}^{2}}{4}\leq\langle x,\frac{A+A^{T}}{2}x\rangle\\
\Rightarrow\max_{x\perp w}\langle x,AA^{T}x\rangle & \leq\max_{x\perp w}\langle x,\frac{A+A^{T}}{2}x\rangle.
\end{align*}
where the last implication followed by the variational characterization
of eigenvalues since $AA^{T}$ and $\frac{1}{2}(A+A^{T})$ are symmetric,
and thus 
\[
\lambda_{2}(AA^{T})\leq\lambda_{2}\left(\frac{A+A^{T}}{2}\right)
\]
which from equation \ref{eq:proof-mixtphi-2}, gives

\[
\tau_{\epsilon}(A)\leq\frac{2\cdot\ln\left(\frac{n}{\kappa\cdot\epsilon}\right)}{1-\sigma_{2}^{2}(A)}=\frac{2\cdot\ln\left(\frac{n}{\kappa\cdot\epsilon}\right)}{1-\lambda_{2}(AA^{T})}\leq\frac{2\cdot\ln\left(\frac{n}{\kappa\cdot\epsilon}\right)}{1-\lambda_{2}\left(\frac{A+A^{T}}{2}\right)}\leq\frac{4\cdot\ln\left(\frac{n}{\kappa\cdot\epsilon}\right)}{\phi^{2}\left(\frac{A+A^{T}}{2}\right)}=\frac{4\cdot\ln\left(\frac{n}{\kappa\cdot\epsilon}\right)}{\phi^{2}(A)}
\]
where the second last inequality follows from Cheeger's inequality
\ref{thm:gen-Cheeger's-inequalities} for the symmetric matrix $(A+A^{T})/2$. 
\end{proof}

\subsection{\label{sec:Proof-of-mixt-lambda2}Proof of Lemma \ref{lem:mixt-lambda2}}
\begin{proof}
Since the mixing time (as shown in the proof of Lemma \ref{lem:mixt-phi-rel}
in Appendix \ref{sec:Proof-of-mixt-phi-rel}), eigenvalues, edge expansion,
and value of $\kappa$ for $R$ and $A=D_{u}^{\frac{1}{2}}D_{v}^{-\frac{1}{2}}RD_{u}^{-\frac{1}{2}}D_{v}^{\frac{1}{2}}$
are the same, we provide the bounds for $A$ and the bounds for $R$
follow. We also restrict to nonnegative vectors, the bound for general
vectors follows by the triangle inequality as discussed after Definition
\ref{def:gen-Mixing-time}.

The lower bound on $\tau_{\epsilon}(R)$ follows by combining the
lower bound on $\tau_{\epsilon}(R)$ in terms of $\phi(R)$ in Lemma
\ref{lem:mixt-phi-rel}, and the upper bound on $\phi(R)$ in terms
of $1-\text{Re}\lambda_{2}(R)$ in Theorem \ref{thm:gen-phi-lambda2}.
For the upper bound on $\tau_{\epsilon}(R)$, similar to the proof
of Lemma \ref{lem:mixt-singval}, we have for any nonnegative vector
$x$ with $\langle w,x\rangle=1$, for $A=w\cdot w^{T}+B$, 
\[
\|D_{w}(A^{\tau}-w\cdot w^{T})x\|_{1}\leq\sqrt{\frac{n}{\kappa}}\cdot\|B^{\tau}\|_{2}
\]
and having 
\[
\|B^{\tau}\|_{2}\leq\frac{\epsilon\sqrt{\kappa}}{\sqrt{n}}
\]
is sufficient. Let $T$ be the triangular matrix in the Schur form
of $B$, from the proof of Lemma \ref{lem:T_norm_1_bound} in Appendix
\ref{sec:Proof-of-T-norm-1-bound}, we have that for 
\[
k\geq\dfrac{3.51n+1.385\ln\left(\frac{n}{\delta}\right)}{1-|\lambda_{m}(A)|},
\]
the norm 
\[
\|B^{k}\|_{2}\leq\delta,
\]
and thus setting $\delta=\frac{\epsilon\sqrt{\kappa}}{\sqrt{n}}$,
we get that 
\[
\tau_{\epsilon}(A)\leq\dfrac{3.51n+1.385\ln\left(n\cdot\frac{\sqrt{n}}{\sqrt{\kappa}\cdot\epsilon}\right)}{1-|\lambda_{m}(A)|}.
\]
Further, since $A$ is $\frac{1}{2}$-lazy, $2A-I$ is also nonnegative,
with the same positive left and right eigenvector $w$ for PF eigenvalue
1, thus having largest singular value 1, and thus every eigenvalue
of $2A-I$ has magnitude at most 1. Similar to the proof of Lemma
\ref{lem:real-lambda-bound} in Appendix \ref{sec:Proof-of-real-lambda-bound},
if $\lambda_{r}=a+i\cdot b$ is any eigenvalue of $A$, the corresponding
eigenvalue in $2A-I$ is $2a-1+i\cdot2b$, whose magnitude is at most
1, giving $(2a-1)^{2}+4b^{2}\leq1$ or $a^{2}+b^{2}\leq a$. It further
gives that 
\[
1-|\lambda_{r}|=1-\sqrt{a^{2}+b^{2}}\geq1-\sqrt{1-(1-a)}\geq1-e^{\frac{1}{2}(1-a)}\geq\frac{1}{4}(1-a)
\]
or 
\[
1-|\lambda_{m}|\geq\frac{1}{4}(1-\text{Re}\lambda_{m})\geq\frac{1}{4}(1-\text{Re}\lambda_{2}),
\]
which gives 
\begin{align*}
\tau_{\epsilon}(A) & \leq4\cdot\dfrac{3.51n+1.385\ln\left(n\cdot\frac{\sqrt{n}}{\sqrt{\kappa}\cdot\epsilon}\right)}{1-\text{Re}\lambda_{2}(A)}\\
 & \leq20\cdot\frac{n+\ln\left(\dfrac{1}{\kappa\cdot\epsilon}\right)}{1-\text{Re}\lambda_{2}(A)}
\end{align*}
completing the proof. 
\end{proof}

\subsection{\label{sec:Proof-of-mixt-AAT}Proof of Lemma \ref{lem:tau-A-AAT}}
\begin{proof}
Since $\phi(A)=\phi(M)$ and each have the same left and right eigenvector
$w$ for PF eigenvalue 1, the lower bound on $\tau_{\epsilon}(A)$
in terms of $\tau_{\epsilon}(M)$ follows immediately from Lemma \ref{lem:mixt-phi-rel},
since 
\[
\sqrt{\tau_{\epsilon}(M)}\leq\ln^{\frac{1}{2}}\left(\frac{n}{\kappa\cdot\epsilon}\right)\cdot\frac{1}{\phi(M)}=\ln^{\frac{1}{2}}\left(\frac{n}{\kappa\cdot\epsilon}\right)\cdot\frac{1}{\phi(A)}\leq\ln^{\frac{1}{2}}\left(\frac{n}{\kappa\cdot\epsilon}\right)\cdot\frac{2\tau_{\epsilon}(A)}{1-2\epsilon}.
\]

For the upper bound, we first define a new quantity for \emph{positive
semidefinite }nonnegative matrices $M$ with PF eigenvalue 1 and $w$
as the corresponding left and right eigenvector. Let $T_{\epsilon}(M)$
be defined as the smallest number $k$ for which 
\[
\|M^{k}-w\cdot w^{T}\|_{2}=\epsilon.
\]

Since $M$ is symmetric, we can write $M=w\cdot w^{T}+UDU^{*}$ where
the first column of the unitary $U$ is $w$, and the diagonal matrix
$D$ contains all eigenvalues of $M$ except the eigenvalue 1 which
is replaced by 0. Further, since $M$ is positive semidefinite, $\lambda_{2}(M)$
is the second largest eigenvalue of $M$, then for every $i>2$, $0\leq\lambda_{i}(M)\leq\lambda_{2}(M)$.
Thus we have 
\[
\|M^{k}-w\cdot w^{T}\|_{2}=\|UD^{k}U^{*}\|_{2}=\|D^{k}\|_{2}=\lambda_{2}^{k}(M)
\]
and 
\[
T_{\epsilon}(M)=\frac{\ln\left(\frac{1}{\epsilon}\right)}{\ln\left(\frac{1}{\lambda_{2}(M)}\right)}.
\]
Further, for the eigenvector $y$ of $M$ corresponding to $\lambda_{2}$,
we have for $k=T_{\epsilon}(M)$, 
\[
(M^{k}-w\cdot w^{T})y=\lambda_{2}^{k}\cdot y=\epsilon\cdot y.
\]
Setting $x=\frac{y}{\|D_{w}y\|_{1}}$, we have $\|D_{w}x\|_{1}=1$,
and we get 
\[
\|D_{w}(M^{k}-w\cdot w^{T})x\|_{1}=\left\Vert D_{w}\frac{\epsilon\cdot y}{\|D_{w}y\|_{1}}\right\Vert _{1}=\epsilon,
\]
which implies that 
\begin{equation}
\tau_{\epsilon}(M)\geq k=T_{\epsilon}(M),\label{eq:proof-AAT-1}
\end{equation}
since there is some vector $x$ with $\|D_{w}x\|=1$ that has $\|D_{w}(M^{k}-w\cdot w^{T})x\|_{1}=\epsilon$
and for every $t<k$, it is also the case that $\|D_{w}(M^{t}-w\cdot w^{T})x\|_{1}>\epsilon$.

Now we observe that since $A$ is $\frac{1}{2}$-lazy with PF eigenvalue
1, and the same left and right eigenvector $w$ for eigenvalue 1,
we have that $M=\frac{1}{2}(A+A^{T})$ is positive semidefinite. We
continue to bound $\tau_{\epsilon}(A)$ similar to the proof method
(in Appendix \ref{sec:Proof-of-mixt-sing-bound}) used for Lemma \ref{lem:mixt-singval}.
For any $t$ and $x$ with $\|D_{w}x\|_{1}=1$, we have that 
\[
\|D_{w}(A^{t}-w\cdot w^{T})x\|_{1}\leq\sqrt{\frac{n}{\kappa}}\sigma_{2}(A)^{t}=\sqrt{\frac{n}{\kappa}}\sigma_{2}^{2}(A)^{\frac{t}{2}}
\]
and thus, 
\[
\tau_{\epsilon}(A)\leq\frac{2\cdot\ln\left(\frac{\sqrt{n}}{\sqrt{\kappa}\cdot\epsilon}\right)}{\ln\left(\frac{1}{\sigma_{2}^{2}(A)}\right)}
\]
and since $A$ is $\frac{1}{2}$-lazy, as shown in the proof of Lemma
\ref{lem:mixt-phi-rel} in Appendix \ref{sec:Proof-of-mixt-phi-rel},
\[
\sigma_{2}^{2}(A)=\lambda_{2}(AA^{T})\leq\lambda_{2}(M),
\]
giving 
\[
\tau_{\epsilon}(A)\leq\frac{2\cdot\ln\left(\frac{\sqrt{n}}{\sqrt{\kappa}\cdot\epsilon}\right)}{\ln\left(\frac{1}{\lambda_{2}(M)}\right)}=\frac{2\cdot\ln\left(\frac{\sqrt{n}}{\sqrt{\kappa}\cdot\epsilon}\right)}{\ln\left(\frac{1}{\epsilon}\right)}\cdot\frac{\ln\left(\frac{1}{\epsilon}\right)}{\ln\left(\frac{1}{\lambda_{2}(M)}\right)}=\frac{2\cdot\ln\left(\frac{\sqrt{n}}{\sqrt{\kappa}\cdot\epsilon}\right)}{\ln\left(\frac{1}{\epsilon}\right)}\cdot T_{\epsilon}(M)\leq\frac{2\cdot\ln\left(\frac{\sqrt{n}}{\sqrt{\kappa}\cdot\epsilon}\right)}{\ln\left(\frac{1}{\epsilon}\right)}\cdot\tau_{\epsilon}(M)
\]
where the last inequality followed from equation \ref{eq:proof-AAT-1}. 
\end{proof}

\subsection{\label{sec:Proof-of-mixt-contchain}Proof of Lemma \ref{lem:mixt-contchain}}
\begin{proof}
We will first find the mixing time of the operator $\exp\left(\frac{R-I}{2}\right)$,
and the mixing time of the operator $\exp(R-I)$ will simply be twice
this number. By expansion of the $\exp$ function, it follows that
$\exp\left(\frac{R-I}{2}\right)$ has PF eigenvalue 1, and $u$ and
$v$ as the corresponding left and right eigenvectors. Further,

\[
\exp\left(\frac{R-I}{2}\right)=e^{-\frac{1}{2}}\left(I+\sum_{i\geq1}\frac{1}{i!}\frac{R^{i}}{2^{i}}\right)
\]
which is $\frac{1}{2}$-lazy due to the first term since $e^{-\frac{1}{2}}\geq\frac{1}{2}$
and all the other terms are nonnegative. Further, for any set $S$
for which $\sum_{i\in S}u_{i}v_{i}\leq\frac{1}{2}$, let $\delta=\frac{1}{2}$,
then 
\begin{align*}
\phi_{S}\left(\exp\left(\frac{R-I}{2}\right)\right) & =e^{-\delta}\cdot\frac{\langle\mathbf{1}_{S},D_{u}\exp\left(\frac{R}{2}\right)D_{v}\mathbf{1}_{\overline{S}}\rangle}{\langle\mathbf{1}_{S},D_{u}D_{v}\mathbf{1}\rangle}\\
 & =e^{-\delta}\cdot\sum_{i\geq1}\frac{\delta^{i}}{i!}\cdot\frac{\langle\mathbf{1}_{S},D_{u}R^{i}D_{v}\mathbf{1}_{\overline{S}}\rangle}{\langle\mathbf{1}_{S},D_{u}D_{v}\mathbf{1}\rangle}\\
 & \ \ \ \ \text{[since \ensuremath{\langle\mathbf{1}_{S},D_{u}ID_{v}\mathbf{1}_{\overline{S}}\rangle=0}]}\\
 & =e^{-\delta}\cdot\sum_{i\geq1}\frac{\delta^{i}}{i!}\cdot\phi_{S}(R^{i})\\
 & \ \ \ \ \text{[since \ensuremath{R^{i}} also has \ensuremath{u} and \ensuremath{v} as the left and right eigenvectors for eigenvalue 1]}\\
 & \leq e^{-\delta}\cdot\sum_{i\geq1}\frac{\delta^{i}}{i!}\cdot i\cdot\phi_{S}(R)\\
 & \ \ \ \ \text{[using Lemma \ref{lem:pow_bound}]}\\
 & =e^{-\delta}\cdot\delta\cdot\phi_{S}(R)\sum_{i\geq1}\frac{\delta^{i-1}}{(i-1)!}\\
 & =e^{-\delta}\cdot\delta\cdot\phi_{S}(R)\cdot e^{\delta}\\
 & =\delta\cdot\phi_{S}(R)
\end{align*}
and thus, 
\begin{equation}
\phi\left(\exp\left(\frac{R-I}{2}\right)\right)\leq\frac{1}{2}\phi(R).\label{eq:mixt-contchain-1}
\end{equation}
Moreover, considering the first term in the expansion, we get 
\begin{align*}
\phi_{S}\left(\exp\left(\frac{R-I}{2}\right)\right) & =e^{-\delta}\sum_{i\geq1}\frac{\delta^{i}}{i!}\cdot\phi_{S}(R^{i})\geq e^{-\delta}\cdot\delta\cdot\phi_{S}(R)
\end{align*}
or 
\begin{equation}
\phi\left(\exp\left(\frac{R-I}{2}\right)\right)\geq\frac{3}{10}\cdot\phi(R).\label{eq:mixt-contchain-2}
\end{equation}
Since $\exp\left(\frac{R-I}{2}\right)$ has left and right eigenvectors
$u$ and $v$, and is $\frac{1}{2}$-lazy, we get from Lemma \ref{lem:mixt-phi-rel}
and equations \ref{eq:mixt-contchain-1} and \ref{eq:mixt-contchain-2}
that 
\[
\frac{1}{2}\cdot\frac{\frac{1}{2}-\epsilon}{\frac{1}{2}\phi(R)}\leq\tau_{\epsilon}\left(\exp(R-I)\right)\leq2\cdot\frac{4\cdot\ln\left(\frac{n}{\kappa\cdot\epsilon}\right)}{\left(\frac{3}{10}\right)^{2}\phi^{2}(R)}\leq\frac{100\cdot\ln\left(\frac{n}{\kappa\cdot\epsilon}\right)}{\phi^{2}(R)}
\]
giving the result. 
\end{proof}

\section{\label{sec:Intuition-behind-the-construction}Intuition behind the
construction }

The main idea behind the construction in Section \ref{sec:Construction-of-non-expanding}
is to systematically reduce the space of doubly stochastic matrices
under consideration, while ensuring that we still have matrices that
we care about in our space. We explain some of the key ideas that
led to the construction.\\

\noindent \textbf{All eigenvalues 0. }Consider de Bruijn graphs. As
stated in Section \ref{subsec:Known-Constructions}, all their nontrivial
eigenvalues are 0, yet they have small expansion. This suggests the
question whether de Bruijn matrices have the \emph{minimum} expansion
amongst \emph{all} doubly stochastic matrices that have eigenvalue
$0$ (which turned out to be far from the case). As we see in the
first steps in the proof of Lemma \ref{lem:T_norm_1_bound} in Appendix
\ref{sec:Proof-of-T-norm-1-bound}, if for a doubly stochastic matrix
$A$, every nontrivial eigenvalue has magnitude at most $1-c$ for
some constant $c$, then powering just $O(\log n)$ times will make
the diagonal entries inverse polynomially small in magnitude, and
thus it would seem that the matrix should have behavior similar to
matrices with all eigenvalues 0. %
{} Thus, the starting point of the construction is to consider only
doubly stochastic matrices that have all eigenvalues 0. Note that
this also helps us to restrict to real (i.e., orthogonal) diagonalizing
unitaries $U$, and real triangular matrices $T$ in the Schur form.
\\

\noindent \textbf{Schur block that looks like a Jordan block.} The
next observation again comes from the proof of Lemma \ref{lem:T_norm_1_bound}
(in Appendix \ref{sec:Proof-of-T-norm-1-bound}). Note that after
writing $A=J+B$, the main idea was to observe that if $B^{m}\approx0$,
then $\phi(A^{m})\approx\phi(J)$. Since $B^{m}=UT^{m}U^{*}$ where
$T$ is upper triangular, this is equivalent to requiring that $T^{m}\approx0$.
As we have assumed that all nontrivial eigenvalues of $T$ are 0,
it means the entire diagonal of $T$ is 0. How large does $m$ have
to be to ensure that $T^{m}\approx0$? It is shown in Lemma \ref{lem:T_norm_1_bound}
that for $m\approx O(n)$, $T^{m}\approx0$. Observe that for Lemma
\ref{lem:T_norm_1_bound}, $\Omega(n)$ is indeed necessary, considering
the case of nilpotent $T$, with all 0 entries, and 1 on every entry
above the diagonal. But can a doubly stochastic matrix have such a
Schur form? Such matrices would at least be sufficient on the outset,
to ensure that Lemma \ref{lem:T_norm_1_bound} is tight, hopefully
giving us the kind of bound we need. Thus, our second restriction
is the following: we assume that every entry of $T$ is 0, and every
off-diagonal entry (entries $T(i,i+1)$ for $2\leq i\leq n-1$) is
$r$. \textit{\emph{We want to }}\textit{maximize}\textit{\emph{ $r$
and still have some unitary that transforms $T$ to a doubly stochastic
matrix}}\textit{.} Note that $T$ will have only 0 in the entries
of the first row and column since we have diagonalized-out the projection
onto the stochastic eigenvector (i.e., $J$). In any upper triangular
matrix $T$ in which the diagonal has zeros, the off-diagonal entries
\emph{affect} the entries in powers of $T$ the most, since entries
far from the diagonal will become ineffective after a few powers of
$T$. Thus, choosing $T$ with the non-zeros pattern of a Jordan block
will not be far from optimal. \\

\noindent \textbf{Restricting to hermitian unitaries with a uniform
row vector.} Next we make several observations about the unitary $U$,
where $A=J+UTU^{*}$. Note that the first column of $U$ is the vector
$\frac{1}{\sqrt{n}}\mathbf{1}$, since it is an eigenvector of $A$.
The choice of $U$ constrains the value of $r$ that we can use in
$T$ since it must lead to $A$ becoming a doubly stochastic matrix.
From considering the optimal $3\times3$ and $4\times4$ cases, which
although, clearly, do not give any information about relation between
$\phi$ and $\text{Re}\lambda_{2}$ or $\Gamma$ (since it could be
off by large constant factors), we note that they do give information
about the \emph{unitary} $U,$ since it is a \emph{rotation} and is
not a numeric value like $\phi$ or $\text{Re}\lambda_{2}$. From
these cases, we make the crucial (albeit empirical) observation: the
\emph{direction} in which $r$ can be maximized (given our constraints)
is such that the first \emph{row} of $U$ is also $\frac{1}{\sqrt{n}}\mathbf{1}$,
and this is the second constraint we impose on $U$. Given that the
unitaries have the \emph{same }vector in the first row and column,
our next observation (again from considering certain specific cases)
is that since the unitary preserves the all $1$'s eigenvector, and
since $T$ has \emph{exactly the same value }above the diagonal by
construction and all other entries are 0, $U$ is performing only
2 types of \emph{actions }on $T$. This type of feature is provided
by unitaries that are hermitian, and this is the next assumption we
make -- we restrict to unitaries that are hermitian. \\

\noindent \textbf{Optimizing for $\phi$ and $r$. }The restrictions
adopted so far on the unitary $U$ and the Schur matrix $T$ imply
a sequence of inequalities to ensure that $A=J+UTU^{*}$ is doubly
stochastic. Subject to the constraints provided by these inequalities,
we aim to \emph{minimize} $\phi$ and \emph{maximize} $r$. Due to
our restrictions on $U$, the cut $S=\{1\}$ in the resulting matrix
$A$ is special, and we aim to \emph{minimize} the edge expansion
$1-A_{1,1}$ of this cut. With the set of possible values that $r$
can take, we note that a set of extreme points of the resulting optimization
problem of minimizing $1-A_{1,1}$ or maximizing $A_{1,1}$ are obtained
if we \emph{force} the values of all the entries $A_{1,i}$ for $3\leq i\leq n$
to 0. We then maximize $r$ for the resulting matrix (indeed, there
are exactly two possible doubly stochastic matrices at this point),
and the result is the construction given in Section \ref{sec:Construction-of-non-expanding}.
\end{document}